\documentclass[11pt, a4paper,oneside, reqno]{amsart}
\usepackage[english]{babel}
\usepackage{amsmath, amsthm, amsfonts, mathrsfs, amssymb, amscd}
\usepackage{bm}
\usepackage{mathtools}
\mathtoolsset{centercolon}
\usepackage{mathabx}
\usepackage{accents}
\usepackage{cancel}

\usepackage[toc]{appendix}

\usepackage[shortlabels]{enumitem}
\usepackage[backgroundcolor=yellow, colorinlistoftodos,prependcaption,textsize=small,textwidth=25mm]{todonotes}
\usepackage[colorlinks, citecolor = blue, urlcolor={red}]{hyperref}
\usepackage{geometry}
\allowdisplaybreaks

\newtheorem{theorem}{Theorem}
\newtheorem{corollary}[theorem]{Corollary}
\newtheorem{lemma}[theorem]{Lemma}
\newtheorem{proposition}[theorem]{Proposition}

\theoremstyle{remark} 
\newtheorem{remark}[theorem]{Remark}

\theoremstyle{definition} 
\newtheorem{definition}[theorem]{Definition}

\numberwithin{theorem}{section}
\numberwithin{equation}{section}

\def\N{{\mathbb N}}

\def\R{{\mathbb R}}

\def\F{{\mathbb F}}
\def\G{{\mathbb G}}

\def\U{{\mathbb U}}

\renewcommand{\vec}[1]{\bm{#1}}

\newcommand{\Dom}{\mathcal{O}}
\newcommand{\BDom}{\partial\Dom}

\newcommand{\n}{\Vert}

\newcommand{\supp}{\text{\rm supp\,}}

\newcommand{\Dir}{{\rm Dir}}
\newcommand{\ii}{{\rm i}} 

\newcommand{\mc}{\mathcal}

\newcommand{\Hinf}{H^{\infty}}
\newcommand{\RR}{\mathbb{R}}

\newcommand{\CC}{\mathbb{C}}
\newcommand{\NN}{\mathbb{N}}

\newcommand{\OO}{\mathcal{O}}

\renewcommand{\SS}{\mathcal{S}}

\newcommand{\half}{\frac{1}{2}}
\newcommand{\RRdh}{\RR^d_+}
\newcommand{\RRd}{\RR^d}
\newcommand{\Cc}{C_{\mathrm{c}}}
\renewcommand{\d}{\partial}
\newcommand{\del}{\Delta}
\newcommand{\grad}{\nabla}
\newcommand{\delDir}{\del_{\operatorname{Dir}}}

\newcommand{\ph}{\varphi}

\newcommand{\gam}{\gamma}

\newcommand{\Tr}{\operatorname{Tr}}
\newcommand{\ext}{\operatorname{ext}}
\newcommand{\loc}{{\rm loc}}
\renewcommand{\tilde}[1]{\widetilde{#1}}
\renewcommand{\hat}[1]{\widehat{#1}}
\newcommand{\dist}{\operatorname{dist}}

\newcommand{\D}{\mathcal{D}}

\renewcommand{\b}{{\rm b}}
\renewcommand{\n}{\vec{n}}

\DeclareMathAlphabet{\mathpzc}{OT1}{pzc}{m}{it}
\newcommand{\pzcd}{\mathpzc{d}}

\newcommand{\dd}{\hspace{2pt}\mathrm{d}}

\DeclareMathOperator{\UMD}{UMD}

\DeclareMathOperator{\id}{id}

\author[R. Denk]{Robert Denk}
\address[Robert Denk]{Universit\"at Konstanz, Fachbereich Mathematik und Statistik, 78357 Konstanz, Germany}\email{robert.denk@uni-konstanz.de}

\author[F.B. Roodenburg]{Floris B. Roodenburg}
\address[Floris Roodenburg]{Delft Institute of Applied Mathematics\\
Delft University of Technology \\ P.O. Box 5031\\ 2600 GA Delft\\The
Netherlands} \email{f.b.roodenburg@tudelft.nl}

\begin{document}
\title[]{Trace theory for parabolic boundary value problems with rough boundary conditions}

\makeatletter
\@namedef{subjclassname@2020}{
  \textup{2020} Mathematics Subject Classification}
\makeatother

\subjclass[2020]{Primary: 35K20, 46E35; Secondary: 46E40.}
\keywords{Trace space, weights, maximal regularity, heat equation, rough domains, rough boundary data}

\thanks{The second author is supported by the VICI grant VI.C.212.027 of the Dutch Research Council (NWO)}

\begin{abstract}
We characterise the trace spaces arising from intersections of weighted, vector-valued Sobolev spaces, where the weights are powers of the distance to the boundary. These weighted function spaces are particularly suitable for treating boundary value problems where derivatives of the solution blow up at the boundary. As an application of our trace theory, we prove well-posedness for the heat equation with rough inhomogeneous boundary data in Sobolev
spaces of higher regularity in domains of fixed regularity $C^{1,\kappa}$, with $\kappa \in [0,1)$.
\end{abstract}

\maketitle

\setcounter{tocdepth}{1}
\tableofcontents

\section{Introduction}
The analysis of boundary value problems with inhomogeneous boundary data requires a precise understanding of the underlying function spaces and their spatial trace spaces. For instance, consider the heat equation on a bounded domain $\OO\subseteq \RR^d$ with Dirichlet boundary conditions, i.e.,
\begin{equation}\label{eq:intro_heateq}
\left\{
  \begin{aligned}
  &\d_t u -\del u = f\qquad &&t>0,\, x\in \OO,\\
   &u|_{\d\OO}= g\qquad &&t>0,\\
   &u|_{t=0}=0\qquad && x\in \OO.
\end{aligned}\right.
\end{equation}
A classical problem is to find necessary and sufficient conditions for the data $f$ and $g$ such that \eqref{eq:intro_heateq} is well-posed and has a solution 
\begin{equation}\label{eq:intro_solspace}
    u\in W^{1,q}(\RR_+; L^p(\OO))\cap L^q(\RR_+; W^{2,p}(\OO)),\qquad p,q\in(1,\infty),
\end{equation}
that depends continuously on the data. It is well known that $f \in L^q(\RR_+; L^p(\OO))$ is a natural condition, while identifying the optimal function space for the boundary data $g$ is more delicate. This problem has been extensively studied in the literature, first for $p=q$ in \cite{LSU68} and later in \cite{DHP07, JS08, Weid98, Weid02, Weid05} for the more general case $p\neq q$. The case $p\neq q$ is particularly relevant for nonlinear problems, where one wants to exploit the scaling invariance of the problem, see, e.g., \cite{Gi86}. Moreover, we note that the end-point case $q=1$ is studied in \cite{OS20,OS22}. 

For \eqref{eq:intro_heateq} on a bounded $C^2$-domain $\OO$, it turns out that the boundary data $g$ has to be in the intersection space
\begin{equation}\label{eq:intro_tracespace_unweighted}
   F^{1-\frac{1}{2p}}_{q,p}(\RR_+; L^p(\d\OO))\cap L^q(\RR_+; B^{2-\frac{1}{p}}_{p,p}(\d\OO)), 
\end{equation}
where $F_{q,p}^s$ and $B^{s}_{p,p}$ denote the Triebel--Lizorkin and Besov spaces, respectively. The space \eqref{eq:intro_tracespace_unweighted} is the optimal space for the boundary data, meaning that the Dirichlet trace operator from the solution space \eqref{eq:intro_solspace} to the spatial trace space \eqref{eq:intro_tracespace_unweighted} is continuous and surjective. \\

In this paper, we study mapping properties of trace operators in a weighted framework. Specifically, we consider spatial weights of the form $w_{\gam}^{\d\OO}(x):=\dist(x, \d\OO)^{\gam}$ for some suitable $\gam>-1$. These weights are well suited to control boundary singularities of solutions and their derivatives, see \cite{LLRV24,LLRV25, LV18}. As a special case of our main result in Theorem \ref{thm:trace_char_dom}, we obtain the following weighted trace result, which is suitable for treating \eqref{eq:intro_heateq} in the weighted setting.
\begin{theorem}\label{thm:intro_trace}
    Let $p,q\in (1,\infty)$, $k\in \NN_0$, $\kappa\in[0,1)$, $\gam\in ((1-\kappa)p-1, 2p-1)\setminus\{p-1\}$. Furthermore, let $\OO$ be a bounded $C^{1,\kappa}$-domain. Then the trace operator $\Tr^{\d\OO}$ from 
    \begin{equation*}
        W^{1,q}(\RR_+; W^{k,p}(\OO, w_{\gam+kp}^{\d\OO}))\cap L^q(\RR_+; W^{k+2,p}(\OO, w_{\gam+kp}^{\d\OO}))
    \end{equation*}
    to 
    \begin{equation*}
        F^{1-\frac{\gam+1}{2p}}_{q,p}(\RR_+; L^p(\d\OO))\cap L^q(\RR_+, B_{p,p}^{2-\frac{\gam+1}{p}}(\d\OO))
    \end{equation*}
    is continuous and surjective. Moreover, there exists a continuous right inverse $\ext^{\d\OO}$ of $\Tr^{\d\OO}$.
\end{theorem}
In addition, we prove a variant of Theorem \ref{thm:intro_trace} with higher-order traces and more time and space regularity, which is required for treating higher-order parabolic equations with inhomogeneous boundary conditions. Furthermore, our approach also covers the vector-valued setting with temporal weights on bounded time intervals. The proof of Theorem \ref{thm:intro_trace} is based on an extension of the trace theory for weighted anisotropic Triebel--Lizorkin spaces as developed in \cite{JS08, Li20} and intersection representations of these spaces. We note that Theorem \ref{thm:intro_trace} (see also Theorem \ref{thm:trace_char_dom}) advances the existing literature in several respects.
\begin{enumerate}[(i)]
    \item We consider spaces with higher-order spatial Sobolev regularity. This allows us to deal with higher-order regularity of \eqref{eq:intro_heateq}, see Theorem \ref{thm:intro_heateq}. It is important to note that the trace space is \emph{independent} of the smoothness $k$, as is also indicated by trace theory for weighted Sobolev spaces, see, e.g., \cite{KimD07, Ro25}. The weighted setting is particularly convenient for higher-order regularity since it avoids any compatibility conditions in the underlying Banach spaces, which are necessary to obtain, for instance, sectoriality of the Laplacian, see \cite{DD11} and \cite{LLRV24, LLRV25} for a more elaborate discussion. 
    \item The regularity of the domain $\OO$ depends on the weight exponent $\gam$, but \emph{not} on the smoothness $k$. This setting coincides with \cite{KimD07, LLRV25}, where the Dahlberg--Kenig--Stein pullback is used for localising results from the half-space to domains.
    \item In Theorem \ref{thm:trace_char_dom} we deal with higher-order traces $\overline{\Tr}_m^{\d\OO}$, which is in contrast to \cite[Theorem 7.1]{LV18}, where trace theory in the weighted setting was developed for second-order parabolic equations with Dirichlet boundary conditions. The higher-order traces model Dirichlet boundary conditions for higher-order parabolic boundary value problems, such as in \cite{DS15}. In upcoming work, we will use this general trace theory for studying a structurally damped plate equation with rough boundary conditions. In addition, we mention that in \cite{HL22, Li20} higher-order trace spaces are characterised for function spaces with mainly Muckenhoupt weights.
\end{enumerate}

As a first application of our trace theory from Theorem \ref{thm:intro_trace}, we study \eqref{eq:intro_heateq} in the setting of weighted Sobolev spaces. In the case of homogeneous boundary conditions $g=0$, functional analytic properties of the Laplacian on $W^{k,p}(\OO, w_{\gam+kp}^{\d\OO})$ are studied in \cite{LLRV24, LLRV25, LV18}.
Our main result for \eqref{eq:intro_heateq} with inhomogeneous boundary conditions reads as follows. This is a special case of Theorem \ref{thm:heateq_0T}, where for simplicity we omit the temporal weights.
\begin{theorem}\label{thm:intro_heateq}
    Let $p,q\in (1,\infty)$, $k\in \NN_0$, $\kappa\in[0,1)$, $\gam\in ((1-\kappa)p-1, 2p-1)\setminus\{p-1\}$ with $1-\frac{1}{q}\neq \frac{\gam+1}{2p}$. Furthermore, let $\OO$ be a bounded $C^{1,\kappa}$-domain. Then for all
    \begin{equation*}
        f\in L^q(\RR_+; W^{k,p}(\OO, w_{\gam+kp}^{\d\OO}))\quad \text{ and }\quad g\in F^{1-\frac{\gam+1}{2p}}_{q,p}(\RR_+; L^p(\d\OO))\cap L^q(\RR_+, B^{2-\frac{\gam+1}{p}}_{p,p}(\d\OO)),
    \end{equation*}
   such that $g|_{t=0}=0$ if $1-\frac{1}{q}>\frac{\gam+1}{2p}$, there exists a unique solution $$u\in W^{1,q}(\RR_+; W^{k,p}(\OO, w_{\gam+kp}^{\d\OO}))\cap L^q(\RR_+; W^{k+2,p}(\OO, w_{\gam+kp}^{\d\OO}))$$ to \eqref{eq:intro_heateq}. Moreover, this solution $u$ depends continuously on the data $f$ and $g$. 
\end{theorem}
Theorem \ref{thm:intro_heateq} is proved by combining maximal regularity for \eqref{eq:intro_heateq} with homogeneous boundary conditions (see \cite[Section 6.1]{LLRV25}) with the trace result from Theorem \ref{thm:intro_trace}. 
It generalises and unifies earlier results, including \cite[Section 7]{LV18} (for bounded $C^2$-domains and $k=0$) and \cite[Corollary 6.7]{LLRV25} (for homogeneous boundary conditions). The main novelties of Theorem \ref{thm:intro_heateq} are:
\begin{enumerate}[(i)]
\item \emph{Rough domains}: larger values of $\gamma$ allow for less regular domains. Furthermore, the required boundary regularity of $\OO$ is independent of $k$, unlike the unweighted theory, where typically $C^{k+2}$-domains are necessary for localisation arguments (see, e.g., \cite{DHP03, Ev10, KrBook08}). In particular, we note that in the unweighted setting $\gam=0$, Theorem \ref{thm:intro_heateq} allows for a bounded $C^{1,\kappa}$-domain with $\kappa> 1-\frac{1}{p}$. Even in the case $k=0$, this improves the standard $C^2$-condition on the domain which is usually imposed in the existing literature. The reason that we can allow for rougher domains is that we use a localisation procedure based on the Dahlberg--Kenig--Stein pullback similar to \cite{KimD07, LLRV25}.
\item \emph{Rough boundary data}: the boundary regularity required of $g$ is likewise independent of $k$. For large $\gamma$, one can treat a wider class of inhomogeneous boundary data than in the unweighted case, cf. \eqref{eq:intro_tracespace_unweighted}.
\end{enumerate}

Finally, using Theorem \ref{thm:intro_heateq}, well-posedness for the heat equation \eqref{eq:intro_heateq} with nonzero initial data can be obtained as well, see Remark \ref{rem:heateq}.

\subsection*{Outline} The outline of this paper is as follows. In Section \ref{sec:weightedspaces} we introduce the necessary function spaces and their properties. In Section \ref{sec:trace_ani} we first prove characterisations of higher-order trace spaces of anisotropic Triebel--Lizorkin spaces, which are used in Section \ref{sec:trace_RRdh} to obtain the trace spaces of intersections of weighted Sobolev spaces on the half-space. In Section \ref{sec:trace_dom} we apply a localisation procedure to transfer the trace results from the half-space to domains with minimal smoothness assumptions. Finally, in Section \ref{sec:heateq} we apply the trace results to prove well-posedness and higher-order regularity of the heat equation with inhomogeneous boundary conditions.

\subsection*{Notation}
We denote by $\NN_0$ and $\NN_1$ the set of natural numbers starting at $0$ and $1$, respectively. 
For $d\in\NN_1$ the half-space is given by $\RRdh=\RR_+\times\RR^{d-1}$, where $\RR_+=(0,\infty)$ and for $x\in \RRdh$ we write $x=(x_1,\tilde{x})$ with $x_1\in \RR_+$ and $\tilde{x}\in \RR^{d-1}$. 

For two topological vector spaces $X$ and $Y$, the space of continuous linear operators is denoted by $\mc{L}(X,Y)$ and $\mc{L}(X):=\mc{L}(X,X)$. Unless specified otherwise, $X$ will always denote a Banach space with norm $\|\cdot\|_X$ and the dual space is $X':=\mc{L}(X,\CC)$.

We write $f\lesssim g$ (resp. $f\gtrsim g$) if there exists a constant $C>0$, possibly depending on parameters which will be clear from the context or will be specified in the text, such that $f\leq Cg$ (resp. $f\geq Cg$). Furthermore, $f\eqsim g$ means $f\lesssim g$ and $g\lesssim  f$.\\

For an open and non-empty $\OO\subseteq \RR^d$ and $\ell\in\NN_0\cup\{\infty\}$, the space $C^\ell(\OO;X)$ denotes the space of $\ell$-times continuously differentiable functions from $\OO$ to some Banach space $X$. As usual, this space is equipped with the compact-open topology. In the case $\ell=0$ we write $C(\OO;X)$ for $C^0(\OO;X)$. Furthermore, we write $C^\ell_\b(\OO;X)$ for the space of all functions in $f\in C^\ell(\OO;X)$ such that $\d^{\alpha} f$ is bounded on $\OO$ for all multi-indices $\alpha\in \NN_0^d$ with $|\alpha|\leq \ell$. 

Let $\Cc^{\infty}(\OO;X)$ be the space of compactly supported smooth functions on $\OO$ equipped with its usual inductive limit topology. The space of $X$-valued distributions is given by $\mc{D}'(\OO;X):=\mc{L}(\Cc^{\infty}(\OO);X)$. Moreover, $\Cc^{\infty}(\overline{\OO};X)$ is the space of smooth functions with their support in a compact set contained in $\overline{\OO}$.

We denote the Schwartz space by $\SS(\RRd;X)$, and $\SS'(\RRd;X):=\mc{L}(\SS(\RRd);X)$ is the space of $X$-valued tempered distributions. For $f\in \SS(\RRd;X)$ we define the $d$-dimensional Fourier transform $(\mc{F} f)(\xi):=\hat{f}(\xi):=\int_{\RRd} f(x)e^{-\ii x\cdot\xi}\dd x$ for $\xi\in \RR^d$, which extends to $\SS'(\RRd;X)$ by duality.

Finally, for $\theta\in(0,1)$ and a compatible couple $(X,Y)$ of Banach spaces, the complex interpolation space is denoted by $[X,Y]_{\theta}$.

\section{Weighted function spaces}\label{sec:weightedspaces}
Let $\Dom$ be a domain in $\R^{d}$ with non-empty boundary $\BDom$. We call a locally integrable function $w:\OO\to (0,\infty)$ a \emph{weight}. In particular, we will deal with the following classes of weights.
\begin{enumerate}[(i)]
\item \emph{Muckenhoupt $A_p$ weights:} let $p\in(1,\infty)$ and let $\OO$ be $\RRd$ or $\RRdh$. A weight $w$ on $\OO$ belongs to the Muckenhoupt class $A_p(\OO)$ if 
\begin{equation*}
    [w]_{A_p(\OO)}:=\sup_{B}\Big(\frac{1}{|B|}\int_Bw(x)\dd x\Big)\Big(\frac{1}{|B|}\int_B w(x)^{-\frac{1}{p-1}}\dd x\Big)^{p-1}<\infty,
\end{equation*}
where the supremum is taken over all balls $B\subseteq \OO$. Moreover, $A_\infty(\OO)=\bigcup_{p>1}A_p(\OO)$. For a detailed study of Muckenhoupt weights and their properties, we refer to \cite[Chapter 7]{Gr14_classical_3rd}.
 \item \emph{Power weights:} for $\gam \in\RR$ define $w^{\d\OO}_{\gam}$ on $\OO$ by
\begin{equation*}
w^{\BDom}_{\gam}(x) := \dist(x,\d\OO)^{\gam}, \qquad x \in \Dom.
\end{equation*}
Furthermore, we write $w_{\gam}(x) := w_\gam^{\smash{\d\RRdh}}(x)=|x_1|^\gam$ for $x=(x_1, \tilde{x})\in \RRdh$. We note that the power weight $w_\gam(x)=|x_1|^\gam$ is in $A_p(\OO)$ if and only if $\gam\in (-1,p-1)$. \\
\end{enumerate}

For $p \in [1,\infty)$, $w$ a weight and $X$ a Banach space, we define the weighted Lebesgue space $L^p(\Dom,w;X)$ as the Bochner space consisting of all strongly measurable $f\colon \mc{O}\to X$ such that
\begin{equation*}
\|f\|_{L^p(\OO,w;X)} := \Big(\int_{\OO}\|f(x)\|_X^p\:w(x)\dd x \Big)^{1/p}<\infty.
\end{equation*}
Let $w$ be such that $w^{-\frac{1}{p-1}}\in L^1_{\loc}(\OO)$. The $k$-th order weighted Sobolev space for $k \in \N_0$ is defined as
\begin{equation*}
W^{k,p}(\Dom,w;X) := \left\{ f \in \mc{D}'(\Dom;X) : \forall |\alpha| \leq k, \partial^{\alpha}f \in L^p(\Dom,w;X) \right\}
\end{equation*}
equipped with the canonical norm. If $w\equiv 1$, then we simply write $W^{k,p}(\OO;X)$. Furthermore, if $X=\CC$, then we omit this from the notation as well.
\begin{remark}\label{rem:L1loc}
  The local $L^1$ condition for $w^{-\frac{1}{p-1}}$ ensures that all the derivatives $\d^{\alpha}f$ are locally integrable in $\OO$. If $\OO$ is the half-space $\RRdh$ or a bounded domain, then this condition for power weights $w_{\gam}^{\d\OO}$ holds for all $\gam\in\RR$. For $\OO=\RR^d$ the local $L^1$ condition holds only for weights $w_\gam(x)= |x_1|^\gam$ with $\gam\in(-\infty,p-1)$ since functions might not be locally integrable near $x_1=0$, see also \cite[Example~1.7]{KO1984}. This explains why, for example, we cannot employ classical reflection arguments from $\RRdh$ to $\RR^d$ if $\gam>p-1$. 
\end{remark}

Let $p\in(1,\infty)$, $s\in\RR$, $w\in A_p(\RRd)$ and let $X$ be a Banach space. For $f\in \SS'(\RRd;X)$ let $J_s f  = (1-\del)^{s/2}f = \mc{F}^{-1}( (1+|\cdot|^2)^{s/2}\mc{F}f)$ be the \emph{Bessel potential operator}. We define the \emph{weighted Bessel potential space} as the space of all $f\in \SS'(\RRd;X)$ such that
\begin{equation*}
    \|f\|_{H^{s,p}(\RRd,w;X)}:=\|J_s f\|_{L^p(\RRd, w;X)}<\infty.
\end{equation*}

Let $\Phi(\RRd)$ be the set of all \emph{inhomogeneous Littlewood--Paley sequences}, see \cite[Section 14.2.c]{HNVW24}.
For $(\ph_n)_{n\geq 1}\in \Phi(\RRd)$, $p\in (1,\infty)$, $q\in[1,\infty], s\in\RR, w\in A_{\infty}(\RRd)$ and $X$ a Banach space, we define the \emph{weighted Besov and Triebel--Lizorkin spaces} as the spaces of all $f\in \mc{S}'(\RRd;X)$ for which, respectively, 
\begin{align*}
  \|f\|_{B^s_{p,q}(\RRd,w;X)}&:=\|(2^{n s}\ph_n\ast f)_{n\geq 0}\|_{\ell^q(L^p(\RRd, w;X))}<\infty,\\
  \|f\|_{F^s_{p,q}(\RRd,w;X)}&:=\|(2^{n s}\ph_n\ast f)_{n\geq 0}\|_{L^p(\RRd, w;\ell^q(X))}<\infty.
\end{align*}
The definitions of the Besov and Triebel--Lizorkin spaces are independent of the chosen Littlewood--Paley sequence up to an equivalent norm, see \cite[Proposition 3.4]{MV12}. We note that all Besov, Triebel--Lizorkin and Bessel potential spaces embed continuously into $\SS'(\RRd;X)$, see \cite[Section 5.2.1e]{Li14}. Furthermore, $\SS(\RR^d;X)$ embeds continuously into all these spaces and this embedding is dense if $p,q<\infty$, see \cite[Lemma 3.8]{MV12}.

Let $\OO\subseteq \RR^d$ be open. Let $\F\in\{B,F,H\}$, then we define the \emph{restriction/factor space}
\begin{equation*}
  \F(\OO,w;X) := \big\{f\in \mc{D}'(\OO;X): \exists g\in \F(\RRd,w;X), g|_{\OO}=f \big\}
\end{equation*}
endowed with the norm
\begin{equation*}
  \|f\|_{\F(\OO,w;X)} := \inf\{\|g\|_{\F(\RRd,w;X)}:g|_{\OO}=f \}.
\end{equation*}
Similarly, we can define the factor space for weighted Sobolev spaces. If $\OO=\RRdh$ and $w\in A_p(\RRdh)$, then the weighted Sobolev spaces $W^{k,p}(\RRdh, w;X)$ as defined above coincide with their factor spaces, see \cite[Section 5]{LMV17}. The proof is based on the existence of an extension operator (in the sense of \cite[Definition 5.2]{LMV17}) and can also be extended to sufficiently smooth domains, see \cite[Theorem 5.22]{AF03}.

\subsection{Anisotropic function spaces} We introduce anisotropic function spaces, where we mainly focus on Triebel--Lizorkin, Bessel potential and Sobolev spaces. For more details on anisotropic spaces, we refer to \cite{Am09, JS08, Li14}.\\

Let $\ell\in\NN_1$ and $\pzcd=(\pzcd_1,\dots, \pzcd_\ell)\in \NN_1^{\ell}$ such that $d:=|\pzcd|=\sum_{j=1}^\ell \pzcd_j$. The \emph{$\pzcd$-decomposition of $\RRd$} is given by
\begin{equation*}
  \RRd = \RR^{\pzcd_1}\times \cdots \times \RR^{\pzcd_\ell},
\end{equation*}
and for $x\in\RRd$ we write $x=(x_1,\dots, x_{\ell})$ and $x_j=(x_{j,1},\dots, x_{j, \pzcd_j})$ where $x_j\in \RR^{\pzcd_j}$ and $x_{j,i}\in \RR$ for $j\in\{1,\dots, \ell\}$ and $i\in\{1,\dots, \pzcd_j\}$.

Let $\vec{p}=(p_1,\dots, p_\ell)\in [1,\infty)^{\ell}$ and $\vec{w}=(w_1,\dots, w_{\ell})$, where $w_j$ is a weight on $\RR^{\pzcd_j}$ for $j\in \{1,\dots, \ell\}$. Furthermore, let $X$ be a Banach space. The \emph{weighted mixed-norm Lebesgue space} $L^{\vec{p},\pzcd}(\RRd,\vec{w};X)$ is the Bochner space consisting of all strongly measurable $f:\RRd\to X$ such that
\begin{equation*}
  \|f\|_{L^{\vec{p},\pzcd}(\RRd,\vec{w};X)}:=\Big(\int_{\RR^{\pzcd_\ell}}\cdots\Big(\int_{\RR^{\pzcd_1}}\|f(x)\|_X^{p_1}w_{1}(x_1) \dd x_1\Big)^{\frac{p_2}{p_1}}\cdots w_{\ell}(x_{\ell})\dd x_{\ell}\Big)^{\frac{1}{p_\ell}}<\infty.
\end{equation*}

To introduce anisotropic function spaces that measure smoothness as well, we first define the following concepts. Let $\RRd$ be $\pzcd$-decomposed and $\vec{a}\in (0,\infty)^\ell$. For $\lambda>0$ we define the \emph{$(\pzcd,\vec{a})$-anisotropic dilation} $\delta_\lambda^{(\pzcd, \vec{a})}:=(\lambda^{a_1}x_1,\dots, \lambda^{a_\ell}x_{\ell})$ for $x\in \RRd$ and the \emph{$(\pzcd,\vec{a})$-anisotropic distance function}
\begin{equation*}
  |x|_{\pzcd, \vec{a}}:=\Big(\sum_{j=1}^\ell |x_j|^{\frac{2}{a_j}}\Big)^\half,\qquad x\in \RRd.
\end{equation*}

In addition, let $0<A<B<\infty$ and define the set of \emph{anisotropic Littlewood--Paley sequences} $\Phi_{A,B}^{\pzcd, \vec{a}}(\RRd)$ as the set of all $(\ph_n)_{n\geq 0}\subseteq \SS(\RRd)$ such that $\hat{\ph}_0=\hat{\ph}$ and 
\begin{equation*}
  \hat{\ph}_n(\xi) = \hat{\ph}\big(\delta_{2^{-n}}^{(\pzcd, \vec{a})}\xi\big)-\hat{\ph}\big(\delta_{2^{-n+1}}^{(\pzcd, \vec{a})}\xi\big),\qquad \xi\in\RRd, \,n\in \NN_1,
\end{equation*}
where the generating function $\ph\in  \SS(\RRd)$ satisfies $0\leq \hat{\ph}(\xi)\leq 1$ for $\xi\in\RRd$, $\hat{\ph}(\xi)=1$ if $|\xi|_{\pzcd, \vec{a}}\leq A$ and $\hat{\ph}(\xi)=0$ if $|\xi|_{\pzcd, \vec{a}}\geq B$. Furthermore, we define $\Phi^{\pzcd, \vec{a}}(\RR^d):=\bigcup_{0<A<B<\infty} \Phi_{A,B}^{\pzcd, \vec{a}}(\RRd)$.\\

Let $\vec{p}\in [1,\infty)^\ell$, $\vec{a}\in(0,\infty)^\ell$, $q\in[1,\infty]$, $s\in \RR$, $\vec{w}\in \prod_{j=1}^\ell A_\infty(\RR^{\pzcd_j})$ and let $X$ be a Banach space. Let $\ph\in \Phi^{\pzcd,\vec{a}}(\RRd)$ and let $(S_n)_{n\geq 0}\subseteq \mc{L}(\SS'(\RRd;X))$ be the associated family of convolution operators given by $S_nf:=\ph_n \ast f$. The \emph{weighted anisotropic Triebel--Lizorkin space $F^{s,\vec{a}}_{\vec{p},q}(\RRd, \vec{w};X)$} is defined as the space of all $f\in \SS'(\RRd;X)$ such that
\begin{equation*}
  \|f\|_{F^{s,\vec{a}}_{\vec{p},q}(\RRd, \vec{w};X)}:= \|(2^{ns} S_nf)_{n\geq 0}\|_{L^{\vec{p},\pzcd}(\RRd,\vec{w};\ell^q(X))}<\infty.
\end{equation*} 
We note that this definition of the Triebel--Lizorkin space is independent of the choice of $\ph\in \Phi^{\pzcd, \vec{a}}(\RRd)$. \\

Next, we define anisotropic Bessel potential and Sobolev spaces. Let $\RRd$ be $\pzcd$-decomposed as above, then for $j\in\{1,\dots, \ell\}$ we define the inclusion and projection maps
\begin{align*}
  \iota_{[\pzcd; j]}&:\RR^{\pzcd_j}\to \RR^d,\qquad x_j\mapsto (0,\dots, 0, x_j, 0,\dots, 0),\\
  \pi_{[\pzcd; j]}&: \RRd\to \RR^{\pzcd_j},\qquad (x_1, \dots, x_\ell)\mapsto x_j.
\end{align*}
For $\sigma\in \RR$ we define the operator $\mc{J}_{\sigma}^{[\pzcd; j]}\in \mc{L}(\SS'(\RRd;X))$ by
\begin{equation*}
  \mc{J}_{\sigma}^{[\pzcd; j]} f: =\mc{F}^{-1}\big((1+|\pi_{[\pzcd;j]}\cdot|^2)^{\frac{\sigma}{2}}\mc{F} f\big).
\end{equation*}
Furthermore, for $\vec{n}=(n_1,\dots, n_\ell)\in \NN_0^\ell$ let $J_{\vec{n},\pzcd}$ be the set of multi-indices given by
\begin{equation*}
  J_{\vec{n},\pzcd}:=\Big\{\alpha \in \bigcup_{j=1}^\ell \iota_{[\pzcd; j]}\NN_0^{\pzcd_j}: |\alpha_j|\leq n_j\Big\}.
\end{equation*}
Let $\vec{p}\in (1,\infty)^\ell$, $\vec{s}\in (0,\infty)^\ell$, $\vec{n}\in \NN_0^\ell$, $\vec{w}\in \prod_{j=1}^\ell A_{p_j}(\RR^{\pzcd_j})$ and let $X$ be a Banach space. Then we define the \emph{weighted anisotropic Bessel potential and Sobolev spaces} as
\begin{align*}
  W^{\vec{n},\vec{p}}_{\pzcd}(\RRd, \vec{w};X) & := \big\{f\in \SS'(\RRd;X): \d^\alpha f\in L^{\vec{p},\pzcd}(\RRd, \vec{w};X),\, \alpha\in J_{\vec{n}, \pzcd} \big\}, \\
 H^{\vec{s},\vec{p}}_{\pzcd}(\RRd, \vec{w};X)  & :=\big\{f\in \SS'(\RRd;X): \mc{J}_{s_j}^{[\pzcd;j]} f\in L^{\vec{p},\pzcd}(\RRd, \vec{w};X),\, j\in\{1,\dots, \ell\}\big\},
\end{align*}
endowed with the norms
\begin{align*}
  \|f\|_{W^{\vec{n},\vec{p}}_{\pzcd}(\RRd, \vec{w};X)} & := \sum_{\alpha\in J_{\vec{n},\pzcd}}\|\d^\alpha f\|_{L^{\vec{p},\pzcd}(\RRd, \vec{w};X)}, \\
  \|f\|_{H^{\vec{s},\vec{p}}_{\pzcd}(\RRd, \vec{w};X)} & := \sum_{j=1}^\ell\|\mc{J}_{s_j}^{[\pzcd;j]} f\|_{L^{\vec{p},\pzcd}(\RRd, \vec{w};X)}.
\end{align*}

\section{Traces of anisotropic Triebel--Lizorkin spaces}\label{sec:trace_ani}
Similar to \cite{JS08, Li14, Li20}, we introduce two notions for the trace at the hyperplane $\{0\}\times \RR^{d-1}$: the distributional trace and the so-called working definition of the trace. \\

\textit{The working definition of the trace.} Let $\ph\in \Phi^{\pzcd, \vec{a}}(\RR^d)$ and let $(S_n)_{n\geq 0}\subseteq \mc{L}(\SS'(\RR^d;X))$ be the associated family of convolution operators given by $S_nf:=\ph_n \ast f$. Let $\mathbb{S}_d\in\{\SS'(\RR^d;X), \SS(\RR^d;X)\}$, then for $f\in \mathbb{S}_d$ it holds that $f=\sum_{n=0}^\infty S_n f$ in $\mathbb{S}_d$ and thus
\begin{equation*}
  f|_{\{0\}\times \RR^{d-1}}=\sum_{n=0}^\infty(S_nf)|_{\{0\}\times \RR^{d-1}} \, \text{ in }\mathbb{S}_{d-1}.
\end{equation*}
By the Paley--Wiener--Schwartz theorem, $S_nf$ is smooth for every $f\in \SS'(\RRd;X)$ and $n\in\NN_0$ and therefore has a well-defined classical trace. We define the trace operator $\Tr: \SS'(\RRd;X)\to \SS'(\RR^{d-1};X)$ by
\begin{equation*}
  \Tr f := \sum_{n=0}^\infty (S_n f)|_{\{0\}\times \RR^{d-1}},
\end{equation*}
whenever this series converges in $\SS'(\RR^{d-1};X)$. \\

\textit{The distributional trace.} Recall that by the Schwartz kernel theorem we have the canonical identification $\D'(\RR\times \RR^{d-1};X)=\D'(\RR; \D'(\RR^{d-1};X))$. Hence, we have
\begin{equation*}
  C(\RR; \D'(\RR^{d-1};X))  \hookrightarrow \D'(\RRd; X)=\D'(\RR\times \RR^{d-1};X)=\D'(\RR; \D'(\RR^{d-1};X)),
\end{equation*}
and we define the bounded evaluation map 
\begin{equation*}
  r: C(\RR; \D'(\RR^{d-1};X)) \to \D'(\RR^{d-1};X),\qquad f\mapsto \operatorname{ev}_0 f,
\end{equation*}
where $\operatorname{ev}_0 f:= f(0)$ denotes the evaluation at $0$.
Since 
\begin{equation*}
  C(\RR^d;X) = C(\RR\times \RR^{d-1};X)= C(\RR; C(\RR^{d-1};X))\hookrightarrow C(\RR; \D'(\RR^{d-1};X)),
\end{equation*}
the classical trace and the distributional trace $r$ coincide on $C(\RR^{d};X)$, i.e.,
\begin{equation*}
  r: C(\RR^d;X) \to C(\RR^{d-1};X),\qquad f\mapsto f|_{\{0\}\times \RR^{d-1}}.
\end{equation*}

\begin{lemma}\label{lem:ext_op_schwartz}
  Let $\ell\in\NN_1$, $\pzcd =(1,\pzcd'') \in\NN_1\times  \NN_1^{\ell-1}$, $\vec{a}=(a_1, \vec{a}'')\in (0,\infty)\times (0,\infty)^{\ell-1}$. Let $\rho\in \SS(\RR)$ with $\rho(0)=1$ and $\supp\hat{\rho}\subseteq [1,2]$ and let $(\phi_n)_{n\geq 0}\in\Phi ^{\pzcd'', \vec{a}''}(\RR^{d-1})$. Then 
  \begin{equation*}
    \ext g(x_1, \tilde{x}):= \sum_{n=0}^\infty \rho (2^{na_1} x_1)  (\phi_n\ast g)(\tilde{x}),\qquad g\in \SS'(\RR^{d-1};X), 
  \end{equation*}
  is a convergent series in $\SS'(\RRd;X)$. Moreover, $$\ext: \SS'(\RR^{d-1};X)\to \SS'(\RR^{d};X)$$ is linear and continuous and 
  $$\ext: \SS'(\RR^{d-1};X)\to C_{\b}(\RR;\SS'(\RR^{d-1};X))$$
  acts as a right inverse of $r: C(\RR; \SS'(\RR^{d-1};X))\to \SS'(\RR^{d-1};X)$.
\end{lemma}
\begin{proof}
All statements except the continuity of $\ext
\colon\SS'(\RR^{d-1};X)\to \SS'(\RR^{d};X)$ are shown in \cite[Proposition~5.2.49]{Li14}, which generalises the scalar case in \cite[Section~4.2]{JS08}. To show the continuity, we use the linear operator $K\colon \SS(\RR^d)\to \SS(\RR^{d-1})$ from \cite[Equation (4.18)]{JS08}, defined by
\[ (K\varphi)(\tilde{x}) := \sum_{n=0}^\infty \int_{\RR}\rho(2^{na_1}x_1) \Big(\int_{\RR^{d-1}} (\tilde{\mathcal{F}}^{-1} \phi_n)(\tilde y)\varphi(x_1,\tilde x-\tilde y)
\dd \tilde y\Big) \dd x_1.\]
Here, $\tilde{\mathcal F}$ stands for the Fourier transform in $\RR^{d-1}$. It is shown in 
\cite[Equation (4.19)]{JS08} that in fact $K\varphi\in 
\SS(\R^{d-1})$ for each $\varphi\in \SS(\RR^d)$ and that
$(\ext g)(\overline\varphi) = g(\overline{K\varphi}) $ for all $g\in\SS'(\RR^{d-1};X)$ and $\varphi\in \SS(\RR^d)$. By the definition of the locally convex
topology on $\SS'(\RR^{d-1};X)$, this shows that  $\ext
\colon\SS'(\RR^{d-1};X)\to \SS'(\RR^{d};X)$ is 
continuous.
\end{proof}

For $m\in\NN_0$ we define the \emph{$m$-th order trace operator} as $\Tr_m:= \Tr\circ \d_1^m$. Furthermore, we define the vector of traces
\begin{equation*}
  \overline{\Tr}_m:=(\Tr_0, \dots, \Tr_m).
\end{equation*}
Moreover, to be consistent with the notation in \cite{JS08,Li14, Li20} we write $\vec{v}=(v_1, \vec{v}'')\in \RR \times \RR^{\ell-1}$ for vectors of length $\ell\in \NN_1$.

\subsection{Trace theory for anisotropic spaces}\label{subsec:ani_trace}
 
We now characterise higher-order trace spaces of weighted anisotropic Triebel--Lizorkin spaces, extending the results in \cite{JS08} (unweighted scalar-valued case) and \cite{Li14, Li20} (weighted vector-valued case for the zeroth-order trace). 
Trace theory for (unweighted) mixed norm spaces was developed in \cite{Ber84, Ber85, Ber87I, Ber87II, Bu71}. In particular, trace theory for anisotropic Triebel--Lizorkin spaces is contained in \cite{FJS00, JS08} and \cite{HL22, Li14, Li20} for the unweighted and weighted case, respectively.\\

By combining the arguments in \cite[Corollary 2.7]{JS08} and \cite[Theorem 4.6]{Li20}, we obtain the characterisation of the higher-order trace space of anisotropic Triebel--Lizorkin spaces. 
We note that a version of the theorem below for $m=0$ and more general weights is contained in \cite[Proposition 3.7]{HL22}.
\begin{theorem}\label{thm:trm_ani_TL}
  Let $\ell\in \NN_1$, $\pzcd_1=1$, $\vec{p}\in (1,\infty)^{\ell}$, $q\in [1,\infty]$, $\vec{a}\in (0,\infty)^{\ell}$, $\gam\in (-1,\infty)$, $m\in \NN_0$ and let $X$ be a Banach space. Let $w_{\gam}(x_1)=|x_1|^\gam$, $\tilde{w}_j\in A_{p_j}(\RR^{\pzcd_j})$ for $j\in\{2,\dots, \ell\}$ and let $\vec{w}=(w_{\gam}, \tilde{w}_2,\dots, \tilde{w}_{\ell})$. Then, for $s> a_1(m+\frac{\gam+1}{p_1})$, the trace operator 
  \begin{equation*}
    \overline{\Tr}_m: F^{s, \vec{a}}_{\vec{p}, q, \pzcd}(\RRd, (w_\gam, \tilde{\vec{w}}'');X)\to \prod_{j=0}^m F^{s-a_1(j+\frac{\gam+1}{p_1}), \vec{a}''}_{\vec{p}'', p_1, \pzcd''}(\RR^{d-1}, \tilde{\vec{w}}'';X)
  \end{equation*}
  is continuous and surjective. Moreover, there exists a continuous right inverse $\overline{\ext}_m$ of $\overline{\Tr}_m$.
  \end{theorem}
  \begin{proof}
    The case $m=0$ is contained in \cite[Theorem 4.6]{Li20} and the formula for the zeroth-order extension operator as given in Lemma \ref{lem:ext_op_schwartz}, restricts to a coretraction corresponding to $\Tr$ on $F^{s,\vec{a}}_{\vec{p}, q,\pzcd}(\RRd, (w_\gam, \tilde{\vec{w}}'');X)$. Now, let $m\in\NN_1$ and $j\in\{0,\dots, m\}$. Note that $\d_1^j$ has order $ja_1$ in the scale of anisotropic Triebel--Lizorkin spaces, see for instance \cite[Proposition 5.2.29]{Li14}. Hence, the continuity for $\Tr_j$ and $\overline{\Tr}_m$ follows from the case $m=0$. 
    
    It remains to consider the extension operator, and we argue similarly as in \cite[Corollary 2.7 \& Section 4.3]{JS08}. Let $(\phi_n)_{n\geq 0}\in \Phi^{\pzcd'', \vec{a}''}(\RR^{d-1})$ and $\rho\in \SS(\RR)$ with $\supp \hat{\rho}\subseteq [1,2]$, $\rho(0)=1$ and $\rho'(0)=\dots= \rho^{(m)}(0)= 0$. Furthermore, let $\rho_j(x_1):= x_1^j \rho(x_1)/ j!$ and  define the higher-order extension operator
    \begin{equation*}
      \ext_j g(x_1, \tilde{x}):=\sum_{n=0}^\infty 2^{-j n  a_1}\rho_j(2^{na_1}x_1) (\phi_n\ast g)(\tilde{x}),\qquad g\in \mc{S}'(\RR^{d-1};X). 
    \end{equation*}
   Due to the particular choice of the $\rho_j$, Lemma \ref{lem:ext_op_schwartz} implies that the series converges in $\mc{S}'(\RR^{d-1};X)$ and that $\ext_j g$ is in $C_\b(\RR; \SS'(\RR^{d-1};X))$ and is continuous as a mapping from $\SS'(\RR^{d-1};X)$ to $\SS'(\RR^{d};X)$. Note that by definition we have for $j_1\leq m$ that $\Tr_{j_1}\rho_j = \delta_{j_1  j}$ where $\delta_{j_1 j}$ is the Kronecker delta. Since $\d_1^{j_1}$ is continuous on $\SS'(\RR^d;X)$ it follows that
   \begin{equation}\label{eq:rightinverse}
     \Tr_{j_1} \ext_j g = \delta_{j_1 j}g, \qquad g\in \mc{S}'(\RR^{d-1};X),\, j_1, j\in \{0,\dots, m\}.
   \end{equation}
     Taking into account the additional power of 2 in the definition of $\ext_j$ (compared to the zeroth-order extension operator), we can redo part (II) of the proof of \cite[Theorem 4.6]{Li20} to obtain that
    \begin{equation*}
      \ext_j: F^{s-a_1(j+\frac{\gam+1}{p_1}), \vec{a}''}_{\vec{p}'', p_1, \pzcd''}(\RR^{d-1}, \tilde{\vec{w}}'';X) \to F^{s, \vec{a}}_{\vec{p}, q, \pzcd}(\RRd, (w_\gam, \tilde{\vec{w}}'');X)
    \end{equation*}
    is continuous.
    Finally, defining
    \begin{equation*}
      \overline{\ext}_m g:= \sum_{j=0}^m \ext_j g_j,\qquad g= (g_0,\dots, g_m)\in \SS'(\RRd;X)^{m+1},
    \end{equation*}
     gives the desired extension operator to $\overline{\Tr}_m$. Indeed, the continuity follows from the continuity properties of $\ext_j$ and $\overline{\Tr}_m \circ \overline{\ext}_m = \id$ follows from \eqref{eq:rightinverse}. 
  \end{proof}

The trace spaces of the weighted anisotropic Triebel--Lizorkin spaces are \emph{independent} of the microscopic parameter $q\in[1,\infty]$. As a corollary, we also obtain the trace space for any normed space $\F$ that satisfies
  \begin{equation*}
    F^{s, \vec{a}}_{\vec{p}, 1, \pzcd}(\RRd, (w_\gam, \tilde{\vec{w}}'');X)\hookrightarrow \F\hookrightarrow F^{s, \vec{a}}_{\vec{p}, \infty, \pzcd}(\RRd, (w_\gam, \tilde{\vec{w}}'');X).
  \end{equation*}
In view of these embeddings (see also \cite[Equation (16)]{Li20}), we obtain the following corollary for the spaces
\begin{equation}\label{eq:ani_WH}
  \begin{aligned}
  \F:= & W^{\vec{n}, \vec{p}}_{\pzcd}(\RRd, (w_\gam, \tilde{\vec{w}}'');X), \quad &\vec{n}&\in \NN_1^\ell, \, \vec{n}=s\vec{a}^{-1}, \\
  \F:= & H^{\vec{s}, \vec{p}}_{\pzcd}(\RRd, (w_\gam, \tilde{\vec{w}}'');X), \quad &\vec{s}&\in (0,\infty)^\ell, \, \vec{s}=s\vec{a}^{-1}.
\end{aligned}
\end{equation}
\begin{corollary}\label{cor:Trm_ani_Sob_Bessel}
  Let $\ell\in \NN_1$, $\pzcd_1=1$, $\vec{p}\in (1,\infty)^{\ell}$, $q\in [1,\infty]$, $\vec{a}\in (0,\infty)^{\ell}$, $\gam\in (-1,p_1-1)$, $m\in \NN_0$ and let $X$ be a Banach space. Let $w_{\gam}(x_1)=|x_1|^\gam$, $\tilde{w}_j\in A_{p_j}(\RR^{\pzcd_j})$ for $j\in\{2,\dots, \ell\}$ and let $\vec{w}=(w_{\gam}, \tilde{w}_2,\dots, \tilde{w}_{\ell})$. Furthermore, let $\F$ be as in \eqref{eq:ani_WH}. Then, for $s> a_1(m+\frac{\gam+1}{p_1})$, the trace operator 
  \begin{equation*}
    \overline{\Tr}_m: \F\to \prod_{j=0}^m F^{s-a_1(j+\frac{\gam+1}{p_1}), \vec{a}''}_{\vec{p}'', p_1, \pzcd''}(\RR^{d-1}, \tilde{\vec{w}}'';X)
  \end{equation*}
  is a continuous and surjective operator. Moreover, there exists a continuous right inverse $\overline{\ext}_m$ of $\overline{\Tr}_m$.
\end{corollary}

\section{Traces of intersections of weighted Sobolev spaces}\label{sec:trace_RRdh}
In this section, we characterise the higher-order trace space of an intersection of weighted Sobolev spaces with spatial domain $\OO=\RRdh$ and with both bounded and unbounded time intervals. We consider spatial and temporal power weights of the form $w_{\gam}(x) = |x_1|^\gam$ and $v_{\mu}(t)= |t|^\mu$, respectively. The main result of this section reads as follows.

\begin{theorem}\label{thm:spatialtracespace}
  Let $p,q\in(1,\infty)$, $k, m\in\NN_0$, $k_1,\ell\in \NN_1$, $\gam\in (-1, (k_1-m)p-1)\setminus\{jp-1: j\in\NN_1\}$, $\mu\in(-1, q-1)$, and let $X$ be a $\UMD$ Banach space. Furthermore, let $I=\RR$ or $I=(0, T)$ for some $T\in(0,\infty]$. Then the trace operator $\overline{\Tr}_{m}$ from
  \begin{equation*}
    W^{\ell,q}(I, v_\mu; W^{k,p}(\RRdh, w_{\gam+kp};X))\cap L^q(I, v_\mu; W^{k+k_1,p}(\RRdh, w_{\gam+kp};X))
  \end{equation*}
  to 
  \begin{equation*}
    \prod_{j=0}^m F_{q,p}^{\ell - \frac{\ell}{k_1}(j+\frac{\gam+1}{p})}(I, v_\mu; L^p(\RR^{d-1};X))\cap L^q(I, v_\mu; B_{p,p}^{k_1-j-\frac{\gam+1}{p}}(\RR^{d-1};X)).
  \end{equation*}
  is continuous and surjective. Moreover, there exists a continuous right inverse $\overline{\ext}_m$ of $\overline{\Tr}_m$.
\end{theorem}
We recall that $X$ is a $\UMD$ (unconditional martingale differences) Banach space if and only if the Hilbert transform extends to a bounded operator on $L^p(\RR;X)$. The $\UMD$ property is a necessary condition for many results on vector-valued spaces. In particular, Hilbert spaces (and thus $\CC$) are $\UMD$ Banach spaces. For more details we refer to \cite[Chapter 4 \& 5]{HNVW16}. 

\begin{remark} We make the following remarks concerning Theorem \ref{thm:spatialtracespace}.
\hspace{2em}
\begin{enumerate}[(i)]
    \item The trace space in Theorem \ref{thm:spatialtracespace} does not depend on the smoothness parameter $k$. This agrees with known trace results for weighted Sobolev spaces, see, e.g., \cite{Ro25}.
    \item If $I=\RR$, then the result of Theorem \ref{thm:spatialtracespace} actually holds for any temporal weight $v\in A_q(\RR)$. The case $I=(0,T)$ follows from using appropriate extension operators, which are easier to construct if $v=v_\mu$ with $\mu\in(-1,q-1)$. We elaborate on the construction of extension operators in Remark \ref{rem:ext_op}.
\end{enumerate}
\end{remark}

Before we turn to the proof of Theorem \ref{thm:spatialtracespace}, we need some additional results. First, we recall the following intersection representation for Triebel--Lizorkin spaces. A more general version can be found in \cite{Li21}, while the unweighted scalar-valued case is contained in \cite[Proposition 3.23]{DK13}. 
\begin{theorem}[{\cite[Theorem 5.2.35]{Li14}}]\label{thm:TL_intersection_repr}
  Let $p,q\in(1,\infty)$, $\vec{a}\in (0,\infty)^2$, $s>0$, $\vec{w}\in A_{p}(\RR^{\pzcd_1})\times A_q(\RR^{\pzcd_2})$ and let $X$ be a Banach space. Then
  \begin{align*}
    F^{s, \vec{a}}_{(p,q),p, \pzcd}(\RRd,\vec{w};X) 
    &=F^{\frac{s}{a_2}}_{q,p}\big(\RR^{\pzcd_2}, w_2; L^p(\RR^{\pzcd_1}, w_1;X)\big)\cap L^q\big(\RR^{\pzcd_2}, w_2;F_{p,p}^{\frac{s}{a_1}}(\RR^{\pzcd_1},w_1;X)\big).
  \end{align*}
\end{theorem}

We can now identify the trace space of an intersection of Sobolev and Bessel potential spaces with Muckenhoupt weights. Additionally, recall that $B^{s}_{p,p}(\RR^{d-1};X)=F^{s}_{p,p}(\RR^{d-1};X)$ for all $s\in \RR$ and $p\in(1,\infty)$, see, for instance \cite[Proposition 14.6.8]{HNVW24}.

\begin{theorem}\label{thm:Trm_H-W_intersection}
  Let $p,q\in(1,\infty)$, $m\in\NN_0$, $n_1,n_2\in \NN_1$, $s_1,s_2\in (0,\infty)$, $v\in A_q(\RR)$, $\gam\in(-1,p-1)$ and let $X$ be a Banach space. If $n_2, s_2> m+ \frac{\gam+1}{p}$, then 
  \begin{align*}
    \overline{\Tr}_m &: W^{n_1,q}(\RR, v; L^p(\RRdh, w_{\gam};X))\cap L^q(\RR, v; W^{n_2,p}(\RRdh, w_{\gam};X)) \\ 
    &\;\to \prod_{j=0}^m F^{n_1-\frac{n_1}{n_2}(j+\frac{\gam+1}{p})}_{q,p}(\RR, v; L^p(\RR^{d-1};X))\cap L^q(\RR, v; B^{n_2-j-\frac{\gam+1}{p}}_{p,p}(\RR^{d-1};X))
  \end{align*}
  and 
  \begin{align*}
    \overline{\Tr}_m &: H^{s_1,q}(\RR, v; L^p(\RRdh, w_{\gam};X))\cap L^q(\RR, v; H^{s_2,p}(\RRdh, w_{\gam};X)) \\ 
    &\;\to \prod_{j=0}^m F^{s_1-\frac{s_1}{s_2}(j+\frac{\gam+1}{p})}_{q,p}(\RR, v; L^p(\RR^{d-1};X))\cap L^q(\RR, v; B^{s_2-j-\frac{\gam+1}{p}}_{p,p}(\RR^{d-1};X))
  \end{align*}
  are continuous and surjective. Moreover, there exists a continuous right inverse $\overline{\ext}_m$ of $\overline{\Tr}_m$.
\end{theorem}
\begin{proof}
  We only consider the case of Sobolev spaces, since the proof for the Bessel potential spaces is completely similar. By a standard restriction argument, it suffices to consider the spatial domain $\RR^d$ instead of $\RRdh$. This is possible as we are in the range of Muckenhoupt weights, and therefore extension operators are available (see \cite[Section 5]{LMV17}). Note that we have the identification
  \begin{align*}
  W^{(n_2,n_1), (p,q)}_{(d,1)}&(\RR^{d+1}, (w_{\gam}, v); X)\\
    =&\;W^{n_1,q}(\RR, v; L^p(\RRd, w_{\gam};X))\cap L^q(\RR, v; W^{n_2,p}(\RRd, w_{\gam};X)),
  \end{align*}
  so that Corollary \ref{cor:Trm_ani_Sob_Bessel} and Theorem \ref{thm:TL_intersection_repr} imply that 
  \begin{align*}
     \overline{\Tr}_m : W&^{n_1,q}(\RR, v; L^p(\RRd, w_{\gam};X))\cap L^q(\RR, v; W^{n_2,p}(\RRd, w_{\gam};X)) \\  \to&\prod_{j=0}^m F^{1-\frac{1}{n_2}(j+\frac{\gam+1}{p}), (\frac{1}{n_2}, \frac{1}{n_1})}_{(p,q),p, (d-1,1)}(\RR^d, (1, v);X)\\
      =&\prod_{j=0}^m F^{n_1-\frac{n_1}{n_2}(j+\frac{\gam+1}{p})}_{q,p}(\RR, v; L^p(\RR^{d-1};X))\cap L^q(\RR, v; B^{n_2-j-\frac{\gam+1}{p}}_{p,p}(\RR^{d-1};X))
  \end{align*}
  is continuous and surjective. Moreover, Corollary \ref{cor:Trm_ani_Sob_Bessel} gives the existence of an extension operator as well.
\end{proof}

We can now characterise the spatial trace space of an intersection of weighted Sobolev spaces as stated in Theorem \ref{thm:spatialtracespace}.

\begin{proof}[Proof of Theorem \ref{thm:spatialtracespace}]
\textit{Step 1: the case $I=\RR$.} Let $v\in A_q(\RR)$.
  Let $r\in \{0, \dots, k_1-m -1\}$ be such that $\gam\in (rp-1, (r+1)p-1)$. By 
  \cite[Proposition 5.5]{LMV17}, \cite[Theorem 3.18]{LV18}, \cite[Proposition 6.6]{Ro25} (using that $X$ is $\UMD$) and Hardy's inequality (see, e.g.,\cite[Corollary 3.4]{LV18}), we obtain
  \begin{align*}
    W^{\ell,q}&\big(\RR, v; W^{k,p}(\RRdh, w_{\gam+kp};X)\big)\cap L^q\big(\RR, v; W^{k+k_1,p}(\RRdh, w_{\gam+kp};X)\big)  \\
    &\hookrightarrow H^{\ell (1-\frac{r}{k_1}),q}\big(\RR, v; [W^{k,p}(\RRdh, w_{\gam+kp};X), W^{k+k_1,p}(\RRdh, w_{\gam+kp};X)]_{\frac{r}{k_1}}\big)\\
    &= H^{\ell (1-\frac{r}{k_1}),q}\big(\RR, v; W^{k+r,p}(\RRdh, w_{\gam+kp};X)\big)\\
    &\hookrightarrow H^{\ell (1-\frac{r}{k_1}),q}\big(\RR, v; L^p(\RRd, w_{\gam-rp};X)\big).
  \end{align*}
By applying Hardy's inequality once more, we have
\begin{align*}
    W^{\ell,q}&(\RR, v; W^{k,p}(\RRdh, w_{\gam+kp};X))\cap L^q(\RR, v; W^{k+k_1,p}(\RRdh, w_{\gam+kp};X))  \\
    & \hookrightarrow H^{\ell (1-\frac{r}{k_1}),q}\big(\RR, v; L^p(\RRd, w_{\gam-rp};X)\big) \cap L^q\big(\RR, v; W^{k_1-r,p}(\RRdh, w_{\gam-rp};X)\big).
\end{align*}
Since $\gam-rp\in (-1,p-1)$ and $k_1-r> m+ \frac{\gam-rp+1}{p}$, Theorem \ref{thm:Trm_H-W_intersection} implies that $\overline{\Tr}_m$ is continuous and surjective from the latter space to
\begin{equation*}
  \prod_{j=0}^m F_{q,p}^{\ell (1-\frac{r}{k_1})-\frac{\ell(1-\frac{r}{k_1})}{k_1-r}(j+\frac{\gam-rp +1}{p})}\big(\RR, v; L^p(\RR^{d-1};X)\big)\cap L^q\big(\RR, v; B_{p,p}^{k_1-r-j-\frac{\gam-rp+1}{p}}(\RR^{d-1};X)\big).
\end{equation*}
which gives the desired trace space after rewriting. This shows that $\overline{\Tr}_m$ is continuous between the desired spaces. 

To prove the existence and continuity of the extension operator, it suffices to establish the embedding
\begin{equation}\label{eq:tracem_ext_emb}
\begin{aligned}
  &F^{1, (a_1, a_2)}_{(p,q),1, (d,1)}(\RRdh\times \RR, (w_\gam,v);X)\\
  &\;\hookrightarrow W^{\ell,q}(\RR, v; W^{k,p}(\RRdh, w_{\gam+kp};X))\cap L^q(\RR, v; W^{k+k_1,p}(\RRdh, w_{\gam+kp};X)),
  \end{aligned}
\end{equation}
with 
\begin{equation}\label{eq:valuesa1a2}
  a_1:= \frac{1}{k+k_1}\quad \text{ and }\quad a_2:=\frac{k_1}{\ell(k+k_1)}.
\end{equation}
Indeed, by Theorems \ref{thm:trm_ani_TL} and \ref{thm:TL_intersection_repr} we have that the extension operator $\overline{\ext}_m$ is continuous from 
\begin{equation*}
    \prod_{j=0}^m F_{q,p}^{\frac{1}{a_2}(1-a_1(j+\frac{\gam+kp+1}{p}))}(\RR, v; L^p(\RR^{d-1};X))\cap L^q(\RR, v; B_{p,p}^{\frac{1}{a_1}(1-a_1(j+\frac{\gam+kp+1}{p}))}(\RR^{d-1};X))
  \end{equation*}
to $F^{1, (a_1, a_2)}_{(p,q),1, (d,1)}(\RRdh\times \RR, (w_\gam,v);X)$. Rewriting using \eqref{eq:valuesa1a2} gives continuity of  $\overline{\ext}_m$ between the desired spaces.

It remains to prove the embedding \eqref{eq:tracem_ext_emb}. Similar as in the proof of \cite[Lemma 7.2]{LV18} it holds that the restriction operator
\begin{equation*}
  R: \mc{D}'(\RRd\times \RR, X)\to \mc{D}'(\RRdh\times \RR;X),\qquad f\mapsto f|_{\RRdh\times \RR}, 
\end{equation*}
restricts to a bounded linear operator on 
\begin{equation*}
  R: F^{0, (a_1, a_2)}_{(p,q),1, (d,1)}(\RRd\times \RR, (w_{\gam},v);X) \to L^q(\RR, v; L^p(\RRdh, w_{\gam};X)).
\end{equation*}
By \cite[Proposition 5.2.29]{Li14}, using that $a_1(k+k_1)=1$ and $a_1k+ a_2 \ell=1$, it follows that 
\begin{equation*}
\begin{aligned}
  &R: F^{1, (a_1, a_2)}_{(p,q),1, (d,1)}(\RRd\times \RR, (w_\gam,v);X)\\
  &\;\to W^{\ell,q}(\RR, v; W^{k,p}(\RRdh, w_{\gam+kp};X))\cap L^q(\RR, v; W^{k+k_1,p}(\RRdh, w_{\gam+kp};X)),
  \end{aligned}
\end{equation*}
is bounded as well. This proves \eqref{eq:tracem_ext_emb} and completes the proof if $I=\RR$.

\textit{Step 2: the case $I=(0, T)$.} This case follows from $I=\RR$ and a suitable extension operator. For the convenience of the reader, we include the details of the proof. First, let $I=(0,\infty)$. By $\overline{\Tr}_m$ and $\overline{\ext}_m$ we denote the (spatial) trace and extension operator from Step 1. For $j\in\{0,\dots, m\}$ we denote by $\Tr^{\RR_+}_j$ the (spatial) trace operator for functions defined on the time interval $I=\RR_+$, i.e., $\Tr_j^{\RR_+} f= f|_{\{0\}\times \RR^{d-1}}$ for $f\in \Cc^\infty(\overline{\RR_+}; \Cc^\infty(\overline{\RRdh};X))$. Furthermore, if $R$ is the restriction operator from $\RR$ to $\RR_+$, then it is clear that
\begin{equation}\label{eq:RTr=TrR}
    R \Tr_j f = \Tr_j^{\RR_+} R f,\qquad j\in\{0,\dots, m\},\, f\in \Cc^\infty(\RR; \Cc^\infty(\overline{\RRdh};X)).
\end{equation}

Let $E$ be a universal extension operator from $\RR_+$ to $\RR$ such that
\begin{equation}\label{eq:ext_F}
      \begin{aligned}
        &E:W^{\ell, q}(\RR_+, v_\mu;Y)\to W^{\ell, q}(\RR, v_\mu;Y)\quad \text{ and }\\
        &E:F_{q,p}^{\ell - \frac{\ell}{k_1}(j+\frac{\gam+1}{p})}(\RR_+, v_\mu;Y) \to F_{q,p}^{\ell - \frac{\ell}{k_1}(j+\frac{\gam+1}{p})}(\RR, v_\mu; Y)
    \end{aligned}
\end{equation}
are bounded for all parameters $p,q, k_1,\ell, j, \gam$ as specified in the theorem and Banach spaces $Y$. Such an extension operator exists by redoing the proof of \cite[Theorem 2.9.2]{Tr83} using the multiplier result \cite[Theorem 1.3]{MV15}. Note that one has to redo the proof of \cite[Theorem 1.3]{MV15} for the Triebel--Lizorkin and Sobolev spaces separately, since in the vector-valued case with $\UMD$ spaces the Sobolev scale is not contained in the Triebel--Lizorkin scale, see, e.g., \cite[Theorem 14.7.9]{HNVW24}.

Moreover, it holds that $RE f = f$ for $f\in \Cc^\infty(\overline{\RR_+}; \Cc^\infty(\overline{\RRdh};X))$. 
Hence, with \eqref{eq:RTr=TrR} it follows that 
\begin{equation*}
    \Tr_j^{\RR_+} f = R\Tr_j E f,\qquad j\in\{0,\dots, m\},\, f\in \Cc^\infty(\overline{\RR_+}; \Cc^\infty(\overline{\RRdh};X)).
\end{equation*}
By the properties of $R$ and $E$, Step 1, and density of $\Cc^\infty(\overline{\RR_+}; \Cc^\infty(\overline{\RRdh};X))$ in 
\begin{equation*}
    W^{\ell,q}(\RR_+,v_\mu; W^{k,p}(\RRdh, w_{\gam+ kp};X)) \cap L^q(\RR_+, v_\mu; W^{k+k_1,p}(\RRdh, w_{\gam+kp};X)),
\end{equation*}
we find that $\overline{\Tr}_m^{\RR_+}$ extends to an operator as stated in the theorem and we have 
    \begin{align*}
        \sum_{j=0}^m\|&\Tr_j^{\RR_+} f \|_{ F_{q,p}^{\ell - \frac{\ell}{k_1}(j+\frac{\gam+1}{p})}(\RR_+, v_\mu; L^p(\RR^{d-1};X))\cap L^q(\RR_+, v_\mu; B_{p,p}^{k_1-j-\frac{\gam+1}{p}}(\RR^{d-1};X))}\\
       & \lesssim \sum_{j=0}^m\|\Tr_j E f\|_{ F_{q,p}^{\ell - \frac{\ell}{k_1}(j+\frac{\gam+1}{p})}(\RR, v_\mu; L^p(\RR^{d-1};X))\cap L^q(\RR, v_\mu; B_{p,p}^{k_1-j-\frac{\gam+1}{p}}(\RR^{d-1};X))}\\
       & \lesssim\|E f\|_{W^{\ell,q}(\RR,v_\mu; W^{k,p}(\RRdh, w_{\gam+ kp};X)) \cap L^q(\RR, v_\mu; W^{k+k_1,p}(\RRdh, w_{\gam+kp};X))}\\
        &\lesssim \|f\|_{W^{\ell,q}(\RR_+,v_\mu; W^{k,p}(\RRdh, w_{\gam+ kp};X)) \cap L^q(\RR_+, v_\mu; W^{k+k_1,p}(\RRdh, w_{\gam+kp};X))}.
    \end{align*}
Define the extension operator by
\begin{equation*}
    \overline{\ext}_m^{\RR_+}:= R \circ \overline{\ext}_m\circ E.
\end{equation*}
Then again by properties of $R$ and $E$, and the continuity of the extension operator $\overline{\ext}_m$, we find the continuity of $\overline{\ext}_m^{\RR_+}$ as stated in the theorem. It remains to prove that this extension is the right inverse to the above constructed trace operator $\overline{\Tr}^{\RR_+}_m$.
Indeed, for $i,j\in \{0,\dots, m\}$ we have using \eqref{eq:RTr=TrR} twice and $RE=\id$ that
\begin{align*}
    \Tr^{\RR_+}_j \ext_i^{\RR_+} h& = R \Tr_j E R\ext_i E h
    = R \Tr_j \ext_i E h\\& = \delta_{i,j} RE h = \delta_{i,j}h,\quad h\in \Cc^\infty(\overline{\RR_+}; \Cc^\infty(\RR^{d-1};X)),
\end{align*}
    and by density this gives the desired result.

Finally, for $I=(0,T)$ with $T\in (0,\infty)$ the result follows similarly by using a universal extension operator from $(0,T)$ to $\RR$. To obtain such an extension operator one can proceed as in \cite[Proposition 2.5]{AV22}.
\end{proof}

\begin{remark}\label{rem:ext_op} There are alternative methods to obtain an extension operator satisfying \eqref{eq:ext_F}.
\begin{enumerate}[(i)]
    \item  The Rychkov extension operator from \cite{Ry99} can be used also in the   weighted and vector-valued setting.
    \item  One can also use the more classical extension operator for Sobolev spaces from \cite[Theorem 5.19]{AF03} (see \cite[Lemma 5.1]{LMV17} for the weighted setting). The boundedness of this extension operator on Triebel--Lizorkin spaces is then obtained by the $\ell^q$-interpolation method, see \cite{LL24}. With this extension operator we expect that one can treat more general (Muckenhoupt) weights, instead of only power weights $v_{\mu}$ with $\mu\in(-1,q-1)$.
\end{enumerate}
\end{remark}

\section{Localisation of traces to rough domains}\label{sec:trace_dom}
We continue with localising the trace result from Theorem \ref{thm:spatialtracespace} to bounded domains. Localisation procedures for smooth domains are quite standard in the literature, see, e.g., \cite[Section 3.6.1]{Tr78}. Since we deal with domains with low regularity, we provide the details of the localisation argument, which is based on the Dahlberg--Kenig--Stein pullback as used in, e.g., \cite{KimD07, KK04, LLRV25}. We first collect the required properties of this diffeomorphism in Section \ref{subsec:loc}. In Section \ref{subsec:trace_dom}, we apply the localisation to prove our main trace theorem on domains.

\subsection{Localisation and traces on domains}\label{subsec:loc} 
We start with some definitions regarding smoothness of domains.\\

Let $\kappa\in (0,1]$ and let $\OO\subseteq \RR^{d-1}$ be open. A function $h:\OO\to \RR$ is called \emph{uniformly $\kappa$-H\"older continuous on $\OO$} if 
\begin{equation*}
  [h]_{\kappa, \OO}:= \sup_{\substack{x, y\in \OO\\ x\neq y}}\frac{|h(x)-h(y)|}{|x-y|^\kappa}<\infty.
\end{equation*}
In addition, for $r\in \NN_0$ we define the space of $\kappa$-H\"older continuous functions by 
\begin{equation*}
  C_{\b}^{r,\kappa}(\OO):=\{f\in C_{\b}^{r}(\OO): [\d^\alpha h]_{\kappa, \OO}<\infty\text{ for all }|\alpha|= r\}.
\end{equation*}
For $\kappa=0$ we write $C_\b^{r,0}(\OO):=C^{r}_\b(\OO)$.  
By $\Cc^{r,\kappa}(\OO)$ we denote the subset of functions in $C^{r,\kappa}(\OO)$ with compact support in $\OO$. Moreover, on $C_\b^{r,\kappa}(\OO)$ we define the norm
\begin{equation*}
  \|h\|_{C^{r,\kappa}(\OO)}:= \sum_{|\alpha|\leq r}\sup_{x\in\OO}|\d^\alpha h(x)|+ \sum_{|\alpha|=r}[\d^\alpha h]_{\kappa, \OO}.
\end{equation*}

\begin{definition}\label{def:domains}
  Let $\Dom \subseteq \R^d$ be a domain, i.e., a connected open set. Let $r \in \N_0$ and $\kappa\in [0,1]$.
  \begin{enumerate}[(i)]
    \item We call $\mc{O}$ a \emph{special $\Cc^{r,\kappa}$-domain} if, after translation and rotation, it is of the form
        \begin{equation}\label{eq:specialdomainh}
          \mc{O} = \{(x_1,\tilde{x})\in \R^d: x_1>h(\tilde{x})\}
        \end{equation}
        for some $h \in \Cc^{r,\kappa}(\R^{d-1};\RR)$.
    \item Given a special $\Cc^{r,\kappa}$-domain $\mc{O}$, we define
\begin{equation*}
[\mc{O}]_{C^{r,\kappa}}:= \|h\|_{C^{r,\kappa}(\R^{d-1})},
\end{equation*}
where $h\in \Cc^{r,\kappa}(\R^{d-1};\RR)$ is such that, after rotation and translation, \eqref{eq:specialdomainh} holds. Note that $[\mc{O}]_{C^{r,\kappa}}$ is uniquely defined due to the compact support of $h$.
    \item We call $\mc{O}$ a \emph{$C^{r,\kappa}$-domain} if every boundary point $x \in \BDom$ admits an open neighbourhood $V$ such that
        \begin{equation*}
        \mc{O}\cap V = W \cap V \qquad \text{and}\qquad \BDom \cap V = \partial W \cap V
        \end{equation*}
        for some special $\Cc^{r,\kappa}$-domain $W$.
  \end{enumerate}
  If $\kappa=0$, then we write $C^r$ for $C^{r,0}$ in the definitions above.
\end{definition}

Next, we introduce a localisation procedure for rough domains. Let $\OO$ be a special $\Cc^{r, \kappa}$-domain with $r\in\NN_1$ and $\kappa\in[0,1]$. Then by \cite[Lemma A.4]{LLRV25}, there exists a $C^{r, \kappa}$-diffeomorphism $\Phi:\OO\to \RRdh$ which straightens the boundary and preserves the distance to the boundary. Moreover, this diffeomorphism is smooth in the interior, but higher-order derivatives may blow up near the boundary. This diffeomorphism is known as the \emph{Dahlberg--Kenig--Stein pullback} and is used in \cite{LLRV25} to obtain boundedness of the $\Hinf$-calculus for the Laplacian on rough domains.
In particular, $\Phi$ satisfies the following mapping properties.
\begin{proposition}[{\cite[Proposition 3.7]{LLRV25}}]\label{prop:isom}
  Let $p\in(1,\infty)$, $r\in \NN_1$, $\kappa\in [0,1]$, $k\in\NN_1$ such that $k\geq r+\kappa$ and $\gam\in ((k-(r+\kappa))p-1,\infty)\setminus\{jp-1:j\in\NN_1\}$. Furthermore, let $\OO$ be a special $\Cc^{r,\kappa}$-domain and let $X$ be a Banach space. Let $\Phi:\OO\to \RRdh$ be the Dahlberg--Kenig--Stein pullback and consider the change of coordinate mappings
  \begin{align*}
  \Phi_*&:W^{k,p}(\OO, w_{\gam}^{\d\OO};X)\to W^{k,p}(\RRdh, w_{\gam};X),
  \end{align*}
  defined by $\Phi_* f := f\circ \Phi^{-1}$. Then $\Phi_*$ is an isomorphism of Banach spaces for which $(\Phi^{-1})_*$ acts as inverse.
\end{proposition}

Next, let $\OO$ be a bounded $C^{r,\kappa}$-domain with $r\in\NN_1$ and $\kappa\in[0,1]$. The following statements hold, see also \cite[Lemma 3.12]{LLRV25}.
\begin{enumerate}[(i)]
    \item There exists a finite open cover $(V_n)_{n=1}^N$ of $\d\OO$, together with special $\Cc^{r,\kappa}$-domains $(\OO_n)_{n=1}^N$, such that
      \begin{equation*}
        \mc{O}\cap V_n = \OO_n \cap V_n \qquad \text{and}\qquad \BDom \cap V_n = \partial \OO_n \cap V_n,\quad n\in\{1,\dots, N\}.
        \end{equation*}
        Without loss of generality, we may assume that each $V_n$ is a ball.
        \item There exist $\eta_0\in \Cc^{\infty}(\OO)$ and $\eta_n \in \Cc^{\infty}(V_n)$ for all $n\in\{1,\dots,N\}$ such that $\sum_{n=0}^{N}\eta_n^2=1$ on $\OO$ and $\sum_{n=1}^{N}\eta_n^2=1$ on $\d\OO$ (partition of unity).
\end{enumerate}
For $n\in \{1,\dots,N\}$, let $\Phi_n: \OO_n\to \RRdh$ be the Dahlberg--Kenig--Stein pullback as introduced above. For any $f$ defined on $\overline{\OO_n}$, the change of coordinate mapping is $(\Phi_n)_* f:= f\circ \Phi_n^{-1}$. Furthermore, for any $g$ defined on $\d\OO_n$, we write $(\Phi_n(0,\cdot))_* g = g (\Phi^{-1}_n (0, \cdot))$ which defines a function on $\RR^{d-1}$.
Throughout this section, we fix $(V_n)_{n=1}^N$, $(\OO_n)_{n=1}^N$, $(\eta_n)_{n=0}^N$ and $(\Phi_n)_{n=1}^N$ as introduced above.\\

Similar to \cite[Definition 3.6.1]{Tr78}, we define Besov spaces on the boundary.  
Let $p\in(1,\infty)$, $s>0$, $X$ a Banach space. Let $L^p(\d\OO;X)$ be the Lebesgue space with respect to the surface measure. Note that the norms $\|g\|_{L^p(\d\OO;X)}$ and 
\[ \sum_{n=1}^N\|(\Phi_n(0,\cdot))_*\eta_n^2 g\|_{L^p(\RR^{d-1};X)},\qquad g\in L^p(\d\OO;X),\]
are equivalent.
We define 
    \begin{align}\label{eq:def_Besov_dom}
        B^{s}_{p,q}(\d\OO;X):=
        \big\{g\in L^p(\d\OO;X): (\Phi_n(0,\cdot))_* \eta_n^2 g \in B^s_{p,q}(\RR^{d-1};X)\, \forall n\in \{1,\dots, N\}\big\},
    \end{align}
equipped with the norm
\begin{align*}
    \|g\|_{B^s_{p,q}(\d\OO;X)}&:=\sum_{n=1}^N\|(\Phi_n(0,\cdot))_*\eta_n^2 g\|_{B^s_{p,q}(\RR^{d-1};X)},\qquad g\in B^s_{p,q}(\d\OO;X).
\end{align*}
Note that this definition is independent of the partition of unity and the diffeomorphisms up to an equivalent norm. Moreover, we need the following spaces of smooth functions on the boundary. Let $L^1_{\loc}(\d\OO;X)$ be the Lebesgue space of locally integrable functions with respect to the surface measure. For $r\in \NN_1$, $\kappa\in[0,1)$ and $\OO$ a bounded $C^{r,\kappa}$-domain, we define
\begin{equation*}
    C^{r,\kappa}(\d\OO; X): = \big\{g\in L^1_{\loc}(\d\OO;X): (\Phi_n(0,\cdot))_* \eta_n^2 g \in C^{r,\kappa}(\RR^{d-1}; X)\, \forall n\in \{1,\dots, N\} \big\}.
\end{equation*}

\subsection{Traces of intersections of weighted Sobolev spaces on domains}\label{subsec:trace_dom}
We first define the spatial trace operator for smooth functions on domains. Let $r\in\NN_1$ and let $\OO$ be a special $\Cc^r$-domain or a bounded $C^r$-domain. Furthermore, the unit inner normal vector on $\d\OO$ is denoted by $\vec{n}$. For $u\in C^{r}(\overline{\OO};X)$ we define the trace
\begin{equation}\label{eq:def_trace_dom}
    \begin{aligned}
    \Tr_j^{\d\OO}u  &= \sum_{|\alpha|=j} (\d^\alpha u)|_{\d\OO} \,\vec{n}^\alpha,\qquad j\in\{0,\dots, r\},\\
     \overline{\Tr}_r^{\d\OO}u& =(\Tr_0^{\d\OO}u  , \dots, \Tr_r^{\d\OO}u ).
\end{aligned}
\end{equation}

The main trace result for intersections of weighted Sobolev spaces on domains reads as follows. The rest of this section is devoted to its proof.
\begin{theorem}\label{thm:trace_char_dom}
    Let $p,q\in(1,\infty)$, $k, m\in\NN_0$, $k_1,\ell, r\in \NN_1$ and $\kappa\in[0,1)$ such that $k_1\geq r+\kappa>m$ and $\gam\in ((k_1-(r+\kappa))p-1, (k_1-m)p-1)\setminus\{jp-1:j\in\NN_1\}$, $\mu\in(-1,q-1)$ and let $X$ be a $\UMD$ Banach space. Furthermore, let $I=\RR$ or $I=(0,T)$ for some $T\in(0,\infty]$ and let $\OO$ be a bounded $C^{r, \kappa}$-domain. Then the trace operator $\overline{\Tr}_m^{\d\OO}$
   from
  \begin{equation*}
    W^{\ell,q}(I, v_\mu; W^{k,p}(\OO, w^{\d\OO}_{\gam+kp};X))\cap L^q(I, v_\mu; W^{k+k_1,p}(\OO, w^{\d\OO}_{\gam+kp};X))
  \end{equation*}
  to 
  \begin{equation}\label{eq:tracespaceO}
    \prod_{j=0}^m F_{q,p}^{\ell - \frac{\ell}{k_1}(j+\frac{\gam+1}{p})}(I, v_\mu; L^p(\d\OO;X))\cap L^q(I, v_\mu; B_{p,p}^{k_1-j-\frac{\gam+1}{p}}(\d\OO;X)).
  \end{equation}
  is continuous and surjective. Moreover, there exists a continuous right inverse $\overline{\ext}_m^{\d\OO}$ of $\overline{\Tr}_m^{\d\OO}$.
\end{theorem}

First of all, it should be noted that $\Tr_1^{\d\OO}$ coincides with the normal derivative $\d_{\n}=\frac{\d}{\d \n}$, while for $j\geq 2$ the trace $\Tr_j^{\d\OO}$ is in general not equal to $\d_{\vec{n}}^j$, see \cite[Remark 2.9]{KimD07}. Nevertheless, one can construct a vector field $\tilde{\n}$ which does satisfy this property. Indeed, if $\OO$ is a special $\Cc^{r,\kappa}$-domain with  $r\geq 2$ and  $\kappa\in[0,1]$, then by \cite[Lemma 2.17]{KimD07} there exists an extension $\tilde{\vec{n}}\in C^{r-1,\kappa}(\OO)$ of $\vec{n}$ such that $\tilde{\vec{n}}=\vec{n}$ on $\d\OO$, which satisfies for $u\in C^{r,\kappa}(\OO;X)$
\begin{equation*}
    \frac{\d^j u}{\d\tilde{\vec{n}}^j} =\sum_{|\alpha|=j}(\d^\alpha u )\,\tilde{\n}^\alpha\qquad  \text{on }\d\OO, \, j\in\{1,\dots, r\}. 
\end{equation*}

To prove Theorem \ref{thm:trace_char_dom}, we use a covering of $\d\OO$ and a partition of unity to reduce to special domains $(\OO_n)_{n=1}^N$. Then, on each special domain, the Dahlberg--Kenig--Stein pullback $\Phi_n$ can be applied to reduce again to the half-space, where we have a trace theorem available. However, the direction of the normal vector $\n$ is not preserved under the Dahlberg--Kenig--Stein pullback. Therefore, we need a trace result on $\RRdh$ with a different non-tangential vector field $\mc{N}$ on the boundary. We first construct this vector field $\mc{N}$ that corresponds to $\n$ under the Dahlberg--Kenig--Stein pullback.  \\

Let  $\OO$ be a special $\Cc^{r,\kappa}$-domain with $r\in\NN_1$ and $\kappa\in [0,1]$. Furthermore, let $m<r+\kappa$ and let $\n$ be the unit inner normal vector on $\d\OO$. 
If $r=1$, set $\tilde{\n}(x)=\n(h(\tilde{x}), x_2, \dots, x_d)$ for $x\in \overline{\OO}$. If $r\geq 2$, then take $\tilde{\n}$ from \cite[Lemma 2.17]{KimD07} as explained above. Using the Dahlberg--Kenig--Stein pullback $\Phi:\OO\to\RRdh$ introduced in Section \ref{subsec:loc}, we define the vector field $\mc{N}\in C^{r-1,\kappa}(\RRdh)$ on $\overline{\RRdh}$ as 
\begin{equation}\label{eq:defmcN}
    \mc{N}(y) = (D\Phi)(\Phi^{-1}(y))\tilde{\n}(\Phi^{-1}(y)),\qquad y\in\overline{\RRdh},
\end{equation}
where $D\Phi$ denotes the Jacobian matrix of $\Phi$. 
Furthermore, $\mc{N}(0, \cdot)$ is a non-tangential vector field on $\d\RRdh=\RR^{d-1}$, which satisfies (see \cite[Equation (19)]{KimD07})  
\begin{equation}\label{eq:trace_v}
    \sum_{|\alpha|=j}(\d^\alpha u)(x)\n^\alpha(x) = \sum_{|\alpha|=j} (\d^\alpha u)(x)\tilde{\n}^\alpha(x) = \frac{\d^j v}{\d \mc{N}^j}(\Phi(x))\quad \text{on }\d\OO,
\end{equation}
for $j\in\{0,\dots, m\}$, $u\in C^{r,\kappa}(\overline{\OO};X)$, and $v(y)=u(\Phi^{-1}(y))$.\\

We will prove a variant of Theorem \ref{thm:spatialtracespace} on $\RRdh$ 
for the vector field $\mc{N}$ on the boundary as defined in \eqref{eq:defmcN}.
\begin{proposition}\label{prop:trace_charN}
    Let $p,q\in(1,\infty)$, $k, m\in\NN_0$, $k_1,\ell, r\in \NN_1$ and $\kappa\in[0,1)$ such that $k_1\geq r+\kappa>m$ and $\gam\in ((k_1-(r+\kappa))p-1, (k_1-m)p-1)\setminus\{jp-1:j\in\NN_1\}$. Furthermore, let $v\in A_{q}(\RR)$ and let $X$ be a $\UMD$ Banach space. Let $\mc{N}\in C^{r-1, \kappa}(\RRdh)$ be the vector field as in \eqref{eq:defmcN}. Then the trace operator 
    \begin{equation*}
        \overline{\Tr}_m^{\mc{N}}:=(\id, \d_\mc{N},\dots, \d_{\mc{N}}^m)|_{\d\RRdh}
    \end{equation*}
   from
  \begin{equation*}
    W^{\ell,q}(\RR, v; W^{k,p}(\RRdh, w_{\gam+kp};X))\cap L^q(\RR, v; W^{k+k_1,p}(\RRdh, w_{\gam+kp};X))
  \end{equation*}
  to 
  \begin{equation}\label{eq:tracespaceN}
    \prod_{j=0}^m F_{q,p}^{\ell - \frac{\ell}{k_1}(j+\frac{\gam+1}{p})}(\RR, v; L^p(\RR^{d-1};X))\cap L^q(\RR, v; B_{p,p}^{k_1-j-\frac{\gam+1}{p}}(\RR^{d-1};X)).
  \end{equation}
  is continuous and surjective. Moreover, there exists a continuous right inverse $\overline{\ext}_m^{\mc{N}}$ of $\overline{\Tr}_m^{\mc{N}}$.
\end{proposition}
\begin{remark}\hspace{2em}
\begin{enumerate}[(i)]
    \item In addition to the conditions in Theorem \ref{thm:spatialtracespace}, there is a more restrictive lower bound on $\gam$ to ensure that the diffeomorphism $\Phi$ has the required mapping properties.
    \item More generally, the proposition actually holds for any non-tangential vector field $\mc{N}$ on $\d\RRdh=\RR^{d-1}$. We will only need the result for the vector field $\mc{N}$ obtained after applying the pullback $\Phi$ to the normal vector $\tilde{\vec{n}}$ at $\d\OO$, see \eqref{eq:defmcN}.
    \item Due to the low regularity of the domain $\OO$ and the vector field $\mc{N}$, we cannot apply the argument as given in \cite[Remark 3.6.1/3]{Tr78} to obtain the result of Proposition \ref{prop:trace_charN}.
\end{enumerate}
\end{remark}
\begin{proof}
\textit{Step 1: a preliminary estimate.} Let $\tilde{\alpha}\in \NN_0^{d-1}$ and $j\in\{0,\dots, m\}$. Then by Theorem \ref{thm:TL_intersection_repr} twice and the fact that the (spatial) $\d^{\tilde{\alpha}}$-derivatives are of order $\frac{|\tilde{\alpha}|}{k_1}$ in the scale of anisotropic Triebel--Lizorkin spaces (e.g., \cite[Proposition 5.2.29]{Li14}), we obtain the estimate
\begin{equation}\label{eq:step1}
\begin{aligned}
        \|\d^{\tilde{\alpha}}g\|&_{F_{q,p}^{\ell-\frac{\ell}{k_1}(j+\frac{\gam+1}{p})}(\RR, v; L^p(\RR^{d-1};X))\cap L^q(\RR, v; B_{p,p}^{k_1-j-\frac{\gam+1}{p}}(\RR^{d-1};X))}\\
        \lesssim & \; \|\d^{\tilde{\alpha}}g\|_{F^{1-\frac{1}{k_1}(j+\frac{\gam+1}{p}), (\frac{1}{k_1}, \frac{1}{\ell})}_{(p,q), p, (d-1, 1)}(\RRd, (1, v);X)}\\
    \lesssim & \; \|g\|_{F^{1-\frac{1}{k_1}(j-|\tilde{\alpha}|+\frac{\gam+1}{p}), (\frac{1}{k_1}, \frac{1}{\ell})}_{(p,q), p, (d-1, 1)}(\RRd, (1, v);X)}\\
        \lesssim &\; \|g\|_{F^{\ell-\frac{\ell}{k_1}(j-|\tilde{\alpha}|+\frac{\gam+1}{p})}_{q,p}(\RR, v; L^p(\RR^{d-1};X))\cap L^q(\RR, v; B^{k_1-j+|\tilde{\alpha}|-\frac{\gam+1}{p}}_{p,p}(\RR^{d-1};X))},
\end{aligned}
\end{equation}
for all $g$ such that the right-hand side is finite.

\textit{Step 2: the trace operator.} We prove the continuity of the trace. By density it suffices to take $u\in \Cc^{\infty}(\RR; \Cc^{\infty}(\overline{\RRdh};X))$. Let $j\in\{0,\dots, m\}$ and let $\mc{N}\in C^{r-1,\kappa}(\RRdh)$ be the non-tangential vector field from \eqref{eq:defmcN}, then by the Leibniz rule we obtain that $\d_{\mc{N}}^j u=(\mc{N}\cdot \grad )^ju$ is a linear combination of terms of the form
  \begin{equation*}
      (\d^\alpha u)\prod_{i=1}^j(\d^{\beta_i}\mc{N}_{r_i})\quad\text{ with } 0\leq \sum_{i=1}^j|\beta_i|\leq j-1 \text{ and } |\alpha|+\sum_{i=1}^j|\beta_i|=j,
  \end{equation*}
  where $\alpha=(\alpha_1, \tilde{\alpha})\in \NN_0\times \NN_0^{d-1}$ satisfies $1\leq |\alpha|\leq j$ and $r_i\in \{1,\dots, d\}$ for all $i\in\{1,\dots, j\}$.
  Moreover, it  holds that $(\d^{\beta_i}\mc{N}_{r_i})(0,\cdot)\in C^{r-j,\kappa}(\RR^{d-1})$ for all $i\in\{1,\dots, j\}$. In particular, it follows that  $(\d^{\beta_i}\mc{N}_{r_i})(0,\cdot)$ are pointwise multipliers on $L^p(\RR^{d-1};X)$ and $B_{p,p}^{k_1-j-\frac{\gam+1}{p}}(\RR^{d-1};X)$ by \cite[Proposition 5.4]{MV15} (using that $\kappa\in[0,1)$) since $r-j+\kappa > k_1-j - \frac{\gam+1}{p}$.
  This, together with \eqref{eq:step1} and Theorem \ref{thm:spatialtracespace}, gives
  \begin{align*}
      \|\d_\mc{N}^j u& |_{\d\RRdh}\|_{F_{q,p}^{\ell - \frac{\ell}{k_1}(j+\frac{\gam+1}{p})}(\RR, v; L^p(\RR^{d-1};X))\cap L^q(\RR, v; B_{p,p}^{k_1-j-\frac{\gam+1}{p}}(\RR^{d-1};X))}\\
      &\lesssim \sum_{|\alpha|\leq j} \|\d^{\tilde{\alpha}}(\d_1^{\alpha_1} u)(0,\cdot)\|_{F_{q,p}^{\ell - \frac{\ell}{k_1}(j+\frac{\gam+1}{p})}(\RR, v; L^p(\RR^{d-1};X))\cap L^q(\RR, v; B_{p,p}^{k_1-j-\frac{\gam+1}{p}}(\RR^{d-1};X))}\\
      &\lesssim \sum_{|\alpha|\leq j} \|(\d_1^{\alpha_1} u)(0,\cdot)\|_{F_{q,p}^{\ell - \frac{\ell}{k_1}(\alpha_1+\frac{\gam+1}{p})}(\RR, v; L^p(\RR^{d-1};X))\cap L^q(\RR, v; B_{p,p}^{k_1-\alpha_1-\frac{\gam+1}{p}}(\RR^{d-1};X))}\\
      &\lesssim \|u\|_{W^{\ell,q}(\RR, v; W^{k,p}(\RRdh, w_{\gam+kp};X))\cap L^q(\RR, v; W^{k+k_1,p}(\RRdh, w_{\gam+kp};X))},
  \end{align*}
  where we have used that $|\tilde{\alpha}|-j = |\alpha|-\alpha_1-j \leq -\alpha_1$.

  \textit{Step 3: the extension operator.} We first prove that there exists an isomorphism $I$ which is bounded on the trace space \eqref{eq:tracespaceN} and relates $\d_{\mc{N}}$-derivatives to $\d_1$-derivatives. Let $j\in\{0,\dots, m\}$, $u\in \Cc^{\infty}(\RR; \Cc^{\infty}(\overline{\RRdh};X))$ and define
  \begin{equation*}
      h_j:=(\d_1^j u)(0,\cdot)\quad \text{ and }\quad g_j:=(\d^j_{\mc{N}}u)(0,\cdot).
  \end{equation*}
  Again, by the Leibniz rule $g_j$ can be written as a sum of $N_1^j h_j$ and terms of the form $(\d^{\tilde{\alpha}}h_{\alpha_1})\prod_{i=1}^j(\d^{\beta_i}\mc{N}_{r_i})(0,\cdot)$ where $|\alpha|\leq j$ with $\alpha_1\leq j-1$ and $|\alpha|+ \sum_{i=1}^j|\beta_i|=j$. Hence, there exists a lower triangular $(m+1)\times (m+1)$ matrix $I$ with diagonal entries $I_{j,j}= \mc{N}^j_1$ such that $(g_0, \dots, g_m)^\top = I (h_0, \dots, h_m)^\top$. From the estimate in Step 2, it follows that $I$ is bounded on the trace space \eqref{eq:tracespaceN}.
  The invertibility of $I$ follows from the fact that $\mc{N}$ is non-tangential, i.e., $\mc{N}_1\geq c>0$ on $\RR^{d-1}$. The boundedness of $I^{-1}$, which is again a lower triangular matrix, can be proved similarly to the boundedness of $I$ after recursively expressing the $h_j$ in terms of the $g_j$. 

  Let $\overline{\ext}_m$ be the extension operator from Theorem \ref{thm:spatialtracespace} which satisfies $\overline{\Tr}_m \circ \overline{\ext}_m = \id$. We claim that $\overline{\ext}_m^{\mc{N}}:= \overline{\ext}_m \circ I^{-1}$ is the desired extension operator. Indeed, the continuity follows from boundedness of $I^{-1}$ and Theorem \ref{thm:spatialtracespace}. Furthermore, it holds that
  \begin{equation*}
     \overline{\Tr}^{\mc{N}}_m \circ \overline{\ext}_m^{\mc{N}} = I\circ  \overline{\Tr}_m\circ \overline{\ext}_m\circ I^{-1} = \id.
  \end{equation*}
  This completes the proof.
\end{proof}
We can now prove our main trace theorem.
\begin{proof}[Proof of Theorem \ref{thm:trace_char_dom}]
We first prove the theorem for $I=\RR$ and in this case we can consider general Muckenhoupt weights $v\in A_q(\RR)$.

Let  $(V_n)_{n=1}^N$, $(\OO_n)_{n=1}^N$, $(\eta_n)_{n=0}^N$ and $(\Phi_n)_{n=1}^N$ be a finite covering of $\d\OO$ consisting of balls, the corresponding special $\Cc^{r, \kappa}$-domains, the partition of unity and the diffeomorphisms as introduced in Section \ref{subsec:loc}, respectively. Furthermore, let $\mc{N}_n$ be the non-tangential vector fields corresponding to $\Phi_n$ as constructed in \eqref{eq:defmcN}.

    \textit{Step 1: the trace operator.} We prove continuity of the trace. By density it suffices to take $u\in \Cc^\infty(\RR; \Cc^{r,\kappa}(\overline{\OO};X))$. Set $v_n: =(\Phi_n)_*\eta_n^2 u$ for $n\in\{1, \dots, N\}$. By Propositions \ref{prop:trace_charN} and \ref{prop:isom} (using that $\gam> (k_1-(r+\kappa))p-1$), we obtain for all $n\in\{1, \dots, N\}$ and $j\in\{0, \dots, m\}$ that
\begin{align*}
    \big\|(\d^j_{\mc{N}_n} v_n)(0,& \cdot)\big\|_{F_{q,p}^{\ell - \frac{\ell}{k_1}(j+\frac{\gam+1}{p})}(\RR, v; L^p(\RR^{d-1};X))\cap L^q(\RR, v; B_{p,p}^{k_1-j-\frac{\gam+1}{p}}(\RR^{d-1};X))}\\
    \lesssim &\;\|v_n\|_{W^{\ell,q}(\RR, v; W^{k,p}(\RRdh, w_{\gam+kp};X))\cap L^q(\RR, v; W^{k+k_1,p}(\RRdh, w_{\gam+kp};X))}\\
     =&\;\|(\Phi_n)_*\eta_n^2 u\|_{W^{\ell,q}(\RR, v; W^{k,p}(\RRdh, w_{\gam+kp};X))\cap L^q(\RR, v; W^{k+k_1,p}(\RRdh, w_{\gam+kp};X))}\\
    \lesssim &\;\|\eta_n^2 u\|_{W^{\ell,q}(\RR, v; W^{k,p}(\OO_n, w^{\d\OO_n}_{\gam+kp};X))\cap L^q(\RR, v; W^{k+k_1,p}(\OO_n, w^{\d\OO_n}_{\gam+kp};X))}\\
    \lesssim &\;\| u\|_{W^{\ell,q}(\RR, v; W^{k,p}(\OO, w^{\d\OO}_{\gam+kp};X))\cap L^q(\RR, v; W^{k+k_1,p}(\OO, w^{\d\OO}_{\gam+kp};X))},
\end{align*}
where in the last step we have used that $\eta_n\in \Cc^\infty(V_n)$, so that the weight $w_{\gam+kp}^{\d\OO_n}$ can be replaced by $w_{\gam+kp}^{\d\OO}$, see also \cite[Lemma 3.12(iii)]{LLRV25}.
    Furthermore, by \eqref{eq:trace_v} we have
    \begin{equation*}
        (\d^j_{\mc{N}_n} v_n)(0, \tilde{y}) =\sum_{|\alpha|=j}  (\d^\alpha \eta_n^2 u )(\Phi_n^{-1}(0,\tilde{y})) \n^{\alpha}(\Phi_n^{-1}(0,\tilde{y})),\qquad \tilde{y}\in \RR^{d-1}.
    \end{equation*}
    Hence, by definition of the trace operator \eqref{eq:def_trace_dom} and the spaces on the boundary \eqref{eq:def_Besov_dom}, this proves the desired estimate.

    \textit{Step 2: the extension operator.} Let $g=(g_0, \dots, g_m)$ be a function in the trace space \eqref{eq:tracespaceO}. For all $n\in\{1,\dots, N\}$ and $j\in \{0,\dots, m\}$, set $h_{n,j}:=(\eta_n^2g_j)(\Phi_n^{-1}(0,\cdot))$ and $h_n:=(h_{n,0}, \dots, h_{n,m})$. It holds that
    \begin{equation*}
        h_{n,j}\in F_{q,p}^{\ell - \frac{\ell}{k_1}(j+\frac{\gam+1}{p})}(\RR, v; L^p(\RR^{d-1};X))\cap L^q(\RR, v; B_{p,p}^{k_1-j-\frac{\gam+1}{p}}(\RR^{d-1};X)).
    \end{equation*}
    Furthermore, let $\overline{\ext}_m^{\mc{N}_n}$ be the extension operator from Proposition \ref{prop:trace_charN} and set $v_n:= \overline{\ext}_m^{\mc{N}_n}h_n$. Proposition \ref{prop:trace_charN} implies that
    \begin{equation}\label{eq:est_ext_cont}
            \begin{aligned}
        \|v_n&\|_{W^{\ell,q}(\RR, v; W^{k,p}(\RRdh, w_{\gam+kp};X))\cap L^q(\RR, v; W^{k+k_1,p}(\RRdh, w_{\gam+kp};X))}\\
        &\lesssim \sum_{j=0}^m\|h_{n,j}\|_{F_{q,p}^{\ell - \frac{\ell}{k_1}(j+\frac{\gam+1}{p})}(\RR, v; L^p(\RR^{d-1};X))\cap L^q(\RR, v; B_{p,p}^{k_1-j-\frac{\gam+1}{p}}(\RR^{d-1};X))}\\
        &\lesssim \sum_{j=0}^m\|g_j\|_{F_{q,p}^{\ell - \frac{\ell}{k_1}(j+\frac{\gam+1}{p})}(\RR, v; L^p(\d\OO;X))\cap L^q(\RR, v; B_{p,p}^{k_1-j-\frac{\gam+1}{p}}(\d\OO;X))},
    \end{aligned}
    \end{equation}
     and $(\d_{\mc{N}_n}^j v_n)|_{\d\RRdh}=h_{n,j}$. Additionally, for $n\in \{1, \dots, N\}$ we define the cut-off function $\zeta_n\in \Cc^\infty(\RRd)$ such that $\supp \zeta_n\subseteq V_n$ and $\zeta_n (x)=1$ if $x\in B_n$, where $B_n$ is a ball which is compactly contained in $V_n$ and satisfies $\supp \eta_n\subseteq B_n$. We claim that
     \begin{equation}\label{eq:def_ext}
         \overline{\ext}_m^{\d\OO}g(x):=u(x):=\sum_{n=1}^N \zeta_n(x)v_n(\Phi_n(x)), \qquad x\in \OO, 
     \end{equation}
     defines the desired extension operator. First, note that we may assume that $v_n(y)=0$ if $y\in \overline{\RRdh}\setminus \Phi_n(\OO\cap B_n)$. In particular, $v_n(y)=0$ if $\zeta_n(\Phi^{-1}_n(y))\neq 1$. This, together with \eqref{eq:trace_v} and the properties of $h_n$, yields
\begin{align*}
     \sum_{|\alpha|=j}(\d^\alpha u) (x)\n^\alpha(x)
    &=\sum_{n=1}^N\zeta_n(x)\sum_{|\alpha|=j} \d^\alpha (v_n(\Phi_n(x)))\n^\alpha(x)\\
    &= \sum_{n=1}^N \zeta_n(x)\frac{\d^j v_n}{\d\mc{N}_n}(\Phi(x)) 
    =\sum_{n=1}^N \eta_n^2 g_j = g_j,\qquad x\in \d\OO, \,j\in\{0,\dots, m\}.
\end{align*}
     It remains to prove the continuity of the extension operator. With \eqref{eq:def_ext} and Proposition \ref{prop:isom} (using that $\gam> (k_1-(r+\kappa))p-1$), we obtain
     \begin{align*}
         \|\overline{\ext}_m^{\d\OO}g&\|_{W^{\ell,q}(\RR, v; W^{k,p}(\OO, w^{\d\OO}_{\gam+kp};X))\cap L^q(\RR, v; W^{k+k_1,p}(\OO, w^{\d\OO}_{\gam+kp};X))}\\
         &\lesssim\sum_{n=1}^N \|(\Phi_n^{-1})_*v_n\|_{W^{\ell,q}(\RR, v; W^{k,p}(\OO_n, w^{\d\OO_n}_{\gam+kp};X))\cap L^q(\RR, v; W^{k+k_1,p}(\OO_n, w^{\d\OO_n}_{\gam+kp};X))}\\
         &\lesssim\sum_{n=1}^N \|v_n\|_{W^{\ell,q}(\RR, v; W^{k,p}(\RRdh, w_{\gam+kp};X))\cap L^q(\RR, v; W^{k+k_1,p}(\RRdh, w_{\gam+kp};X))}
     \end{align*}
     and the required estimate follows using \eqref{eq:est_ext_cont}. This completes the proof for $I=\RR$.

     Finally, for $I=(0,T)$ with $T\in (0,\infty]$ the results follow similar to Step 2 in the proof of Theorem \ref{thm:spatialtracespace}. In this case, we should work with $\Cc^\infty(\overline{I}; \Cc^{r,\kappa}(\overline{\OO};X))$ and $\Cc^\infty(\overline{I}; C^{r,\kappa}(\d\OO;X))$ as dense subsets, since the spaces $B^{s}_{p,q}(\d\OO;X)$ and $C^{r,\kappa}(\d\OO;X)$ (see Section \ref{subsec:loc}) are defined using $C^{r,\kappa}$-diffeomorphisms.
\end{proof}

\section{The heat equation on rough domains with rough boundary data}\label{sec:heateq}
In this section, we prove well-posedness and higher-order spatial regularity for the heat equation on rough domains with rough boundary data. This extends the results in \cite[Corollary 6.7]{LLRV25} (homogeneous boundary data) and \cite[Section 7.3]{LV18} (bounded $C^2$-domains and base regularity). We study the initial boundary value problem given by
\begin{equation}\label{eq:heateq_R+}
\left\{
  \begin{aligned}
  &\d_t u -\del u = f\qquad &(t,x)&\in \RR_+\times \OO,\\
   &\Tr^{\d\OO}u= g\qquad &(t,x)&\in \RR_+\times \d\OO,\\
   &u(0)=0&x&\in\OO.
\end{aligned}\right.
\end{equation}
We study this problem with spatial regularity in weighted Sobolev spaces $W^{k,p}(\OO, w_{\gam}^{\d\OO})$, where $k\in\NN_0$ and
\begin{enumerate}[(i)]
    \item $\OO$ is a bounded $C^1$-domain if $\gam\in(p-1,2p-1)$, or,
    \item $\OO$ is a bounded $C^{1,\kappa}$-domain with $\kappa\in (0,1)$ if $\gam\in((1-\kappa)p-1, p-1)$.
\end{enumerate}
This setting is similar to \cite{LLRV25}. For $p,q\in (1,\infty)$, $k\in \NN_0$, $\gam\in (-1,2p-1)\setminus\{p-1\}$, $\mu\in(-1,q-1)$, $v_\mu(t)=|t|^\mu$ and $\OO$ as above, we define 
\begin{align*}
\F^{k}_{q,p}(\RR_+, v_\mu;\OO, w_{\gam+kp}^{\d\OO} )&:=L^q(\RR_+, v_\mu; W^{k,p}(\OO, w_{\gam+kp}^{\d\OO})),\\
   \U^{k}_{q,p}(\RR_+, v_\mu; \OO, w_{\gam+kp}^{\d\OO})&:= W^{1,q}(\RR_+, v_\mu; W^{k,p}(\OO, w_{\gam+kp}^{\d\OO})) \cap L^q(\RR_+, v_\mu; W^{k+2,p}(\OO, w_{\gam+kp}^{\d\OO})),\\
   \G_{q,p}(\RR_+, v_\mu; \d\OO)&:= F^{1-\frac{\gam+1}{2p}}_{q,p}(\RR_+, v_\mu; L^p(\d\OO))\cap L^q(\RR_+, v_\mu; B_{p,p}^{2-\frac{\gam+1}{p}}(\d\OO)).
\end{align*}
Moreover, we introduce
\begin{align*}
    \prescript{}{0}{\G}_{q,p}(\RR_+, v_\mu; \d\OO)&:=\begin{cases}
        \G_{q,p}(\RR_+, v_\mu; \d\OO)&\mbox{ if }1-\frac{\mu+1}{q}<\frac{\gam+1}{2p},\\
        \{g\in \G_{q,p}(\RR_+, v_\mu; \d\OO): g(0)=0\}&\mbox{ if }1-\frac{\mu+1}{q}>\frac{\gam+1}{2p},
    \end{cases}\\
    \prescript{}{0}{\U}^{k}_{q,p}(\RR_+, v_\mu; \OO, w_{\gam+kp}^{\d\OO})& := \{u\in \U^{k}_{q,p}(\RR_+, v_\mu; \OO, w_{\gam+kp}^{\d\OO}): u(0)=0\}.
\end{align*}
The zeroth-order temporal traces in the above spaces are well defined, see, e.g., \cite[Theorem 1.2]{Ro25}.\\

We obtain the following maximal regularity result for \eqref{eq:heateq_R+}.
\begin{theorem}\label{thm:heateq_0T}
    Let $p,q\in (1,\infty)$, $k\in\NN_0$, $\mu\in(-1,q-1)$, $\kappa\in[0,1)$, $\gam\in((1-\kappa)p-1,2p-1)\setminus\{p-1\}$ with $1-\frac{\mu+1}{q}\neq \frac{\gam+1}{2p}$. Furthermore, let $\OO$ be a bounded $C^{1,\kappa}$-domain. Then for all
    \begin{equation*}
        f\in \F^{k}_{q,p}(\RR_+, v_\mu; \OO, w_{\gam+kp}^{\d\OO})\quad \text{ and }\quad g\in \prescript{}{0}{\G}_{q,p}(\RR_+, v_\mu; \d\OO),
    \end{equation*}
   there exists a unique solution $u\in \prescript{}{0}{\U}^{k}_{q,p}(\RR_+, v_\mu; \OO, w_{\gam+kp}^{\d\OO})$ to \eqref{eq:heateq_R+}. Moreover, this solution satisfies
   \begin{equation*}
       \|u\|_{\U^{k}_{q,p}(\RR_+, v_\mu; \OO, w_{\gam+kp}^{\d\OO})}\leq C\big(\|f\|_{\F^{k}_{q,p}(\RR_+, v_\mu; \OO, w_{\gam+kp}^{\d\OO})} + \|g\|_{\G_{q,p}(\RR_+, v_\mu; \d\OO)}\big),
   \end{equation*}
   where the constant $C>0$ only depends on $p,q,k,\mu, \kappa, \gam$, and $d$.
\end{theorem}
\begin{proof}
    From Theorem \ref{thm:trace_char_dom} it follows that 
    \begin{equation*}
        \Tr^{\d\OO}: \U^{k}_{q,p}(\RR, v_\mu; \OO, w_{\gam+kp}^{\d\OO})\to \G_{q,p}(\RR, v_\mu; \d\OO)
    \end{equation*}
    is continuous and surjective. By using the zero extension operator from $\RR_+$ to $\RR$, we can argue similarly as in Step 2 of Theorem \ref{thm:spatialtracespace} to obtain that
    \begin{equation}\label{eq:heateq_trace}
        \Tr^{\d\OO}: \prescript{}{0}{\U}^{k}_{q,p}(\RR_+, v_\mu; \OO, w_{\gam+kp}^{\d\OO})\to \prescript{}{0}{\G}_{q,p}(\RR_+, v_\mu; \d\OO)
    \end{equation}
    is continuous and surjective. The corresponding extension operator is denoted by $\ext^{\d\OO}$. First consider 
    \begin{equation}\label{eq:heateq_hom}
\left\{
  \begin{aligned}
  &\d_t \tilde{u} -\del \tilde{u} = \tilde{f}\qquad &(t,x)&\in \RR_+\times \OO,\\
   &\Tr^{\d\OO}\tilde{u}= 0\qquad &(t,x)&\in \RR_+\times \d\OO,\\
   &\tilde{u}(0)=0 & x&\in \OO,
\end{aligned}\right.
\end{equation}
    where $\tilde{f}:= f - (\d_t -\del)\ext^{\d\OO} g \in \F^{k}_{q,p}(\RR_+, v_\mu; \OO, w_{\gam+kp}^{\d\OO})$. By \cite[Corollary 6.7]{LLRV25}, there exists a solution $$\tilde{u} \in W^{1,q}(\RR_+, v_\mu; W^{k,p}(\OO, w_{\gam+kp}^{\d\OO})) \cap L^q(\RR_+, v_\mu; W_\Dir^{k+2,p}(\OO, w_{\gam+kp}^{\d\OO}))$$ (for the definition of the Sobolev spaces with homogenous Dirichlet boundary conditions we refer to \cite[Section 3.3]{LLRV25}) which satisfies the estimate
    \begin{equation}\label{eq:heateqest_tilde}
        \begin{aligned}
        \|\tilde{u}\|_{\U^{k}_{q,p}(\RR_+, v_\mu; \OO, w_{\gam+kp}^{\d\OO})}&\lesssim \|f\|_{\F^{k}_{q,p}(\RR_+, v_\mu; \OO, w_{\gam+kp}^{\d\OO})} + \|(\d_t-\del)\ext^{\d\OO}g\|_{\F^{k}_{q,p}(\RR_+, v_\mu; \OO, w_{\gam+kp}^{\d\OO})}\\
        &\lesssim \|f\|_{\F^{k}_{q,p}(\RR_+, v_\mu; \OO, w_{\gam+kp}^{\d\OO})} + \|\ext^{\d\OO}g\|_{\U^{k}_{q,p}(\RR_+, v_\mu; \OO, w_{\gam+kp}^{\d\OO})}.
    \end{aligned}
    \end{equation}
    Set $u:=\tilde{u}+ \ext^{\d\OO} g$. Then $u$ satisfies $\Tr^{\d\OO}u = \Tr^{\d\OO}\tilde{u} + \Tr^{\d\OO}\ext^{\d\OO} g =g$ by \eqref{eq:heateq_hom} and properties of the extension operator. Furthermore, by \eqref{eq:heateq_trace} and \eqref{eq:heateq_hom}, we have $u(0)=\tilde{u}(0)+\ext g (0) = 0$. Hence, by the construction of $\tilde{u}$ in \eqref{eq:heateq_hom}, we have that $u$ satisfies \eqref{eq:heateq_R+} and we have the estimate
    \begin{align*}
        \|u\|_{\U^{k}_{q,p}(\RR_+, v_\mu; \OO, w_{\gam+kp}^{\d\OO})}&\lesssim \|f\|_{\F^{k}_{q,p}(\RR_+, v_\mu; \OO, w_{\gam+kp}^{\d\OO})} + \|\ext^{\d\OO}g\|_{\U^{k}_{q,p}(\RR_+, v_\mu; \OO, w_{\gam+kp}^{\d\OO})}\\
        &\lesssim \|f\|_{\F^{k}_{q,p}(\RR_+, v_\mu; \OO, w_{\gam+kp}^{\d\OO})} + \|g\|_{\G_{q,p}(\RR_+, v_\mu; \d\OO)},
    \end{align*}
    using \eqref{eq:heateqest_tilde} and \eqref{eq:heateq_trace}. Finally, the uniqueness of the solution $u$ to \eqref{eq:heateq_R+} follows from \cite[Corollary 6.7]{LLRV25}.
\end{proof}

\begin{remark}\hspace{2em}\label{rem:heateq}
    \begin{enumerate}[(i)]
        \item Well-posedness and regularity for the heat equation with non-zero initial data in a real interpolation space can be obtained as well, see \cite[Section 4.4]{GV17} or \cite[Section 17.2.b]{HNVW24}. In the unweighted setting, characterisations of the real interpolation spaces between Sobolev spaces are well known. However, in the weighted setting, characterisations of real interpolation spaces and temporal trace spaces are more difficult, see \cite[Section 7.2]{LV18}.
        \item In \cite[Section 6]{LLRV25} the results are also valid in the $\UMD$ Banach space-valued case if $-\delDir$ is replaced by $\lambda-\delDir$ with $\lambda\geq \lambda_0$ for some $\lambda_0\geq 0$ depending on the geometry of the Banach space. Hence, a version of Theorem \ref{thm:heateq_0T} can be obtained in the vector-valued setting as well. For the special case $k=0$, a vector-valued version of Theorem \ref{thm:heateq_0T} can be obtained with $\lambda_0=0$, see \cite[Corollary 6.10]{LLRV25}.
        \item Instead of the parabolic problem one could also consider the elliptic problem $\lambda u -\del u = f$ on $\OO$ with $\lambda>0$ and  boundary conditions $u|_{\d\OO}=g$. If $g=0$, then well-posedness and regularity (in the weighted setting on rough domains) follows from \cite[Theorem 1.1]{LLRV25}. Combining this with the trace results on rough domains from \cite{KimD07}, yields well-posedness and regularity for the elliptic problem with inhomogeneous boundary conditions $g\neq 0$.
    \end{enumerate}
\end{remark}

\bibliographystyle{plain}
\bibliography{references_dampedplate}
\end{document}